\documentclass[reqno]{amsart}

\usepackage{array}
\usepackage{amsmath}
\usepackage{amsfonts}
\usepackage{amssymb}
\usepackage{amsthm}
\usepackage{amsmath, amscd}
\usepackage{marvosym}
\usepackage{hyperref}
\usepackage{float}
\usepackage{graphicx}
\newtheorem{introtheorem}{Theorem}
\newtheorem{introcor}{Corollary}
\newtheorem{theorem}{Theorem}[section]
\newtheorem{Theorem}{Theorem}
\newtheorem{lemma}[theorem]{Lemma}
\newtheorem{proposition}[theorem]{Proposition}
\newtheorem{corollary}[theorem]{Corollary}

\newtheorem{definition}[theorem]{Definition}
\newtheorem{remark}[theorem]{Remark}

\newtheorem{conjecture}[Theorem]{Conjecture}

\newcommand{\op}[1]{\operatorname{#1}}
 
\newcommand{\dbcoh}[1]{\op{D}^{\op{b}}_{\op{coh}}(#1)}

\def\Z{\mathop{\mathbb{Z}}}

\def\O{\mathcal{O}}

\def\P{\mathbb{P}}

\def\der{\mathop{\text{D}^{\text{b}}_{\text{coh}}(X)}}
\def\ider{\mathop{\emph{D}^{\emph{b}}_{\emph{coh}}(X)}}
\def\idder{\mathop{\emph{D}^{\emph{b}}_{\emph{coh}}(X \times X)}}
\def\isder{\mathop{\emph{D}^{\emph{b}}_{\emph{coh}}(\mathcal X)}}
\def\isdder{\mathop{\emph{D}^{\emph{b}}_{\emph{coh}}(\mathcal X \times \mathcal X)}}
\def\dery{\mathop{\text{D}^{\text{b}}_{\text{coh}}(Y)}}

\def\coh{\text{coh}}
\def\icoh{\emph{coh}}
\def\perf{\text{perf}}
\def\Hom{\text{Hom}}

\def\ra{\rightarrow}

\def\dd{\text{dim}_{\Delta}}
\def\idd{\emph{dim}_{\Delta}}
\def\tritime{\text{\Clocklogo}}
\def\Ext{\text{Ext}}

\def\L{\mathop{\mathcal{L}}}

\title{Hochschild Dimensions of Tilting Objects}

\author[Ballard]{Matthew Ballard}
\address{
  \begin{tabular}{l}
   Matthew Ballard  \\ 
   \hspace{.1in} University of South Carolina, Department of Mathematics, Columbia, SC, USA \\
   \hspace{.1in} Email: {\bf ballard@math.sc.edu} \\
  \end{tabular}
}

\author[Favero]{David Favero}
\address{
  \begin{tabular}{l}
   David Favero \\
   \hspace{.1in} University of Alberta, Department of Mathematics, Edmonton, AB, Canada \\
   \hspace{.1in} Email: {\bf favero@gmail.com} \\
  \end{tabular}
}
\numberwithin{equation}{section}

\begin{document}

\begin{abstract}
 We give a new upper bound for the generation time of a tilting object and use it to verify, in some new cases, a conjecture of Orlov on the Rouquier dimension of the derived category of coherent sheaves on a smooth variety.
\end{abstract}

\maketitle

\section{Introduction}

In \cite{Ro2}, R. Rouquier introduced a notion of dimension for triangulated categories.  Roughly, the Rouquier dimension is the infimum over all generators of the minimal number of triangles it takes to build the category from a generator.

Under some mild hypotheses on a variety, $X$, Rouquier also showed that the Rouquier dimension of $\der$ is finite, bounded below by the dimension of the variety, and, for a smooth variety, bounded above by twice the dimension of the variety.

The following conjecture is due to D. Orlov \cite{O4}:
\begin{conjecture} \label{conj:1}
Let $X$ be a smooth variety.  The Rouquier dimension of $\ider$ equals the dimension of $X$.
\end{conjecture}

In \cite{Ro2}, Rouquier showed that Conjecture~\ref{conj:1} is true for affine varieties, flag varieties (of type A), and quadrics.  Recently, Orlov proved that this conjecture is true for curves \cite{O4}.

In this paper, we will study the case where $X$ is a smooth variety whose derived category of coherent sheaves possesses a tilting object, $T$. We give a new upper bound on the number of cones needed to build all of $\der$ from $T$. Recall that the Hochschild dimension of a $k$-algebra, $A$, is the projective dimension of $A$ as an $A \otimes_k A^{\op{op}}$-module. $\omega_X$ denotes the canonical bundle of $X$.

\begin{introtheorem} \label{thm:main}
Let $i_0$ be the largest $i$ for which $\emph{Hom}_{X}(T,T \otimes_{\mathcal O_X} \omega_X^{\vee}[i])$ is nonzero. The Hochschild dimension of $\emph{End}_X(T)$ is equal to $\emph{dim}(X) + i_0$. If $i_0$ is zero, then the Hochschild dimension of $\emph{End}_X(T)$, the Rouquier dimension of $\ider$, and the dimension of $X$ are all equal.
\end{introtheorem}

Applying Theorem \ref{thm:main} to examples of varieties (and stacks) known to possess tilting objects, we are able to enlarge the set of varieties for which Conjecture \ref{conj:1} is true. Below we list a handful of examples.

\begin{introcor}
Assume that $\op{char}(k)$ is zero and that $k$ is algebraically-closed. Conjecture \ref{conj:1} holds for:
\begin{itemize}
 \item del Pezzo surfaces with $\emph{rk} \ \emph{Pic}(X) \leq 7$;
 \item Fano threefolds of types $V_5$ and $V_{22}$;
 \item toric surfaces with nef anti-canonical divisor;
 \item toric Deligne-Mumford stacks of dimension no more than two or Picard number no more than two ($k=\mathbb{C}$);
 \item and Hirzebruch surfaces.
\end{itemize}
\end{introcor}

The case of Hirzebruch surfaces is of particular interest. Using Theorem \ref{thm:main}, we show that, for most Hirzebruch surfaces, it takes three cones for any tilting bundle to generate the derived category. However, there is an essentially surjective functor from a weighted projective stack to the Hirzebruch surface. The image of a tilting bundle from the weighted projective stack gives a generator which needs only two cones to build any object.

\vspace{2.5mm}
\noindent \textbf{Acknowledgments:} \emph{We are grateful to Tony Pantev and Dmitri Orlov for useful conversations and correspondence. We would also like to thank Asher Auel, Tobias Dyckerhoff, Umut Isik, and the referee for their careful reading, helpful suggestions, and corrections. This work was funded by NSF Research Training Group Grant, DMS 0636606.}
\vspace{2.5mm}

\section{Preliminaries}
In this section, we recall some of the necessary background and gather the results which will be of importance to us later on. We always work over a base field which we denote by $k$. For now, we place no restrictions on $k$. A variety will refer to a seperated, reduced, and irreducible scheme of finite type over $k$.  On a smooth variety, $X$, we write $\omega_X$ for the canonical bundle and $K$ for the corresponding divisor. In the categories under investigation in this paper, direct sums are coproducts. Consequently, we will denote coproducts by $\oplus$.

\subsection{Dimension of a triangulated category}

Let $\mathcal T$ be a triangulated category. Recall that an object, $Y$, is called a summand of $X$ if there is another object, $Z$, and an isomorphism, $Y \oplus Z \cong X$. For a subcategory, $\mathcal I$, of $\mathcal T$ we denote by $\langle \mathcal I \rangle$ the full subcategory of $\mathcal T$ whose objects are isomorphic to summands of finite coproducts of shifts of objects in $\mathcal I$. In other words, $\langle \mathcal I \rangle$ is the smallest full subcategory containing $\mathcal I$ and closed under isomorphisms, shifting, and taking finite direct coproducts and summands. For two full subcategories, $\mathcal I_1$ and $\mathcal I_2$, we denote by $\mathcal I_1 \ast \mathcal I_2$ the full subcategory of objects, $B$, such that there is a distinguished triangle, $B_1 \to B \to B_2 \to B_1[1]$, with $B_i \in \mathcal I_i$.  Set $\mathcal I_1 \diamond \mathcal I_2 := \langle \mathcal I_1 \ast \mathcal I_2 \rangle$, $\langle \mathcal I \rangle_0 :=\langle \mathcal I \rangle$, and inductively define, 
\begin{displaymath}
 \langle \mathcal I \rangle_n := \langle \mathcal I \rangle_{n-1} \diamond \langle \mathcal I \rangle.
\end{displaymath}
Similarly we define,
\begin{displaymath}
 \langle \mathcal I \rangle_{\infty} := \bigcup_{n \geq 0} \langle \mathcal I \rangle_{n}.
\end{displaymath}
$\langle \mathcal I \rangle_{\infty}$ is the smallest thick subcategory of $\mathcal T$ containing $\mathcal I$. 

The operation, $\diamond$, on subcategories arose in \cite{BV}. The following is Lemma 2.1.1 of loc. cit.
\begin{lemma}\label{lem:associative}
 $\langle \mathcal I \rangle_n \diamond \langle \mathcal I \rangle_m = \langle \mathcal I \rangle_{n+m+1}$.
\end{lemma}
The reader is warned that, in loc. cit. and other previous papers, $\langle \mathcal I \rangle_0 := 0$ and $\langle \mathcal I \rangle_1 := \langle \mathcal I \rangle$. This previous indexing, has the advantage that the above formula becomes $\langle \mathcal I \rangle_n \diamond \langle \mathcal I \rangle_m = \langle \mathcal I \rangle_{m+n}$.  However, with our convention, the index equals the number of cones allowed and will often be equal to other familiar invariants.

We will also require a slight variation which allows for infinite coproducts. Let $\overline{\mathcal I}$ denote the smallest full subcategory of $\mathcal T$ closed under isomorphisms, shifts, summands, and all coproducts.   %In the literature, the category $\overline{\langle \mathcal I \rangle}_n$ is inductively defined by $\overline{\langle \mathcal I \rangle}_n := \overline{\overline{\langle \mathcal I \rangle}_{n-1} \diamond \overline{\mathcal I}}$.  The final overline is redundant, thus $\overline{\langle \mathcal I \rangle}_n = \langle \overline{\mathcal I} \rangle_n$.

\begin{definition} \label{def:gen}
Let $E$ be an object of a triangulated category $\mathcal{T}$.  If there is an $n$ with $\langle E \rangle_{n} = \mathcal T$, we set,
\begin{displaymath}
 \tritime(E):=  \emph{min } \lbrace  n \geq 0 \  | \ \langle E \rangle_{n} = \mathcal T \rbrace.
\end{displaymath}
Otherwise, we set $\tritime(E) := \infty$.   We call $\tritime(E)$ the \textbf{generation time} of $E$. If $\langle E \rangle_{\infty}$ equals $\mathcal{T}$, we say that $E$ is a \textbf{generator}. If $\tritime(E)$ is finite, we say that $E$ is a \textbf{strong generator}. The \textbf{Rouquier dimension} of $\mathcal T$, denoted $\emph{dim }\mathcal T$, is the minimal generation time amongst strong generators. It is set to $\infty$ if there are no strong generators.
\end{definition}

\begin{remark}
 One can also form an invariant that captures all of the information described in Definition \ref{def:gen}. The Orlov spectrum of $\mathcal T$ is the list of all generation times of strong generators of $\mathcal T$. Many open  questions about the Orlov spectra of derived categories of coherent sheaves on smooth varieties exist, but, already, there are hints about deep ties to the geometry of the underlying variety, see \cite{BFK}.
\end{remark}

Let $F: \mathcal{T} \to \mathcal{R}$ be an exact functor between triangulated categories. If every object in $\mathcal{R}$ is isomorphic to a summand of an object in the essential image of $F$, we say that $F$ is dense, or has dense image. We give a few simple but useful lemmas.

\begin{lemma} \label{lem:functor invariant}
 Let $G$ be an object of $\mathcal T$. If $B \in \langle G \rangle_n$, then $F(B) \in \langle F(G) \rangle_n$.  Moreover, if $F$ commutes with coproducts and $B \in {\langle \overline{G} \rangle}_n$, then $F(B) \in \langle \overline{F(G)} \rangle_n$.
\end{lemma}

\begin{proof}
 Any exact functor commutes with finite coproducts and takes exact triangles to exact triangles so $F \left(\langle G \rangle_n\right) \subset \langle F(G) \rangle_n$. To get the identity, $F \left(\langle \overline{G} \rangle_n\right) \subset \langle \overline{F(G)} \rangle_n$, we need to assume that $F$ commutes with all coproducts.
\end{proof}

\begin{lemma} \label{density lemma}
  If $F: \mathcal{T} \to \mathcal{R}$ has dense image, then $\emph{dim }\mathcal{T} \geq \emph{dim }\mathcal{R}$. 
\end{lemma}

\begin{proof}
 If $G$ is a generator of $\mathcal T$ with minimal generation time $t$, then $\mathcal T = \langle G \rangle_t$. We apply $F$ and use Lemma \ref{lem:functor invariant} to get $F(\mathcal T) \subset \langle F(G) \rangle_t$. Since every object of $\mathcal R$ is a summand of an object $F(\mathcal T)$, we see that $\mathcal R = \langle F(G) \rangle_t$. Thus, $\dim \mathcal R \leq t$.
\end{proof}

\begin{lemma} \label{lem:autostrong}
 If $\mathcal{T}$ is a triangulated category with finite Rouquier dimension, then any generator is a strong generator.
\end{lemma}

\begin{proof}
 Let $X$ be a generator of $\mathcal T$.  As $\mathcal{T}$ is finite Rouquier dimensional, there exists a strong generator, $G$, with $\langle G \rangle_n = \mathcal T$. Since $X$ generates, $G \in \langle X \rangle_t$ for some $t$.  Hence $\langle X \rangle_{(n+1)(t+1)-1} = \mathcal T$ by Lemma \ref{lem:associative}.
\end{proof}

Let $k$ be a field and $A$ be a $k$-algebra. An $A$-module, $M$, is coherent if it is finitely-generated and the kernel of any map from a finite rank free module to $M$ is finitely-generated. Recall that coherent $A$-modules form an abelian category.

We will consider the following derived categories associated to $A$:\ $\text{D}(\text{Mod-}A)$, the derived category of unbounded complexes of right $A$-modules; $\text{D}^{\text{b}}(\text{mod-}A)$, the derived category of bounded complexes of coherent right $A$-modules; and $\text{D}_{\text{perf}}(A)$, the perfect derived category of right $A$-modules. Recall that $\text{D}_{\text{perf}}(A)$ is the smallest thick triangulated subcategory generated by the free module $A$ in $\text{D}(\text{Mod-}A)$, i.e.\ $\text{D}_{\text{perf}}(A) \cong \langle A \rangle_{\infty}$.

Let $X$ be a variety over $k$. We denote by $\op{D}_{\op{qcoh}}(X)$ the derived category of quasi-coherent sheaves on $X$, and by $\op{D}^{\op{b}}_{\op{coh}}(X)$ the bounded derived category of coherent sheaves on $X$. Recall that a complex in $\op{D}_{\op{qcoh}}(X)$ is called perfect if it locally, in the Zariski topology, is quasi-isomorphic to a bounded complex of locally-free coherent sheaves. We denote by $\op{D}_{\op{perf}}(X)$ the full subcategory of perfect complexes. 

In each of these cases, we have compactly-generated triagulated categories, $\op{D}(\op{Mod-}A)$ and $\op{D}_{\op{qcoh}}(X)$, where the compact objects are exactly the objects of $\op{D}_{\op{perf}}(A)$ and $\op{D}_{\op{perf}}(X)$. See \cite{Nee} for the definitions of compact objects and compactly-generated triangulated categories and, Theorem 2.1 in particular, for a proof that the categories of perfect objects and compact objects coincide in these examples. Note that requiring that $\op{D}_{\op{perf}}(A)$ is a subcategory of $\text{D}^{\text{b}}(\text{mod-}A)$ is equivalent to requiring that $A$ itself is coherent. As $X$ is assumed to be of finite-type over $k$, $\op{D}_{\op{perf}}(X)$ is a subcategory of $\op{D}^{\op{b}}_{\op{coh}}(X)$.

In algebraic and geometric situations, the Rouquier dimension of a triangulated category is related to common homological invariants, e.g.\ the global dimension and the Hochschild dimension of a $k$-algebra.  For the convenience of the reader, we now recall the definition of the Hochschild dimension and the global dimension of a $k$-algebra.

\begin{definition}
Let $A$ be a $k$-algebra.  The \textbf{Hochschild dimension} of $A$, denoted $\emph{hd}(A)$, is the projective dimension of $A$ as an $A \otimes_k A^{\emph{op}}$-module. The \textbf{global dimension} of $A$, denoted by $\emph{gldim}(A)$, is the supremum over all right $A$-modules, $M$, of the projective dimension of $M$.
\end{definition}

To compress notation, we set $A^e := A \otimes_k A^{\text{op}}$. The categories of left or right $A^e$-modules are equivalent to the category of $A$-bimodules. The vector space $A \otimes_k A$ has many $A^e$-module structures. We shall consider it as an $A^e$-module via the outer bimodule structure, i.e.\ left multiplication on the first copy of $A$ and right multiplication on the second copy of $A$. With this bimodule structure, $A \otimes_k A$ and $A^e$ are isomorphic as left $A^e$-modules. Similarly, $A$ is always taken to have the natural bimodule structure given by left multiplication on the left and right multiplication on the right. If $A$ is a perfect $A^e$-module, the Hochschild dimension of $A$ can be understood as follows:

\begin{lemma}
 Assume $A$ is a perfect $A^e$-module. The Hochschild dimension of $A$ is equal to the minimal $m$ for which $A \in \langle A^e \rangle_m$ in $\emph{D}(A^e\emph{-Mod})$.
\end{lemma}
\begin{proof}
 Since $A$ is a perfect $A^e$-module, we may take a minimal $n$ such that $A$ lies in $\langle A^e \rangle_n$ and denote this by $d$.  Any element of $\op{Ext}_{A^e}^l(A,M)$ is represented by an exact sequence,
\begin{displaymath}
 0 \to M \to M_{l-1} \to \cdots \to M_0 \to A \to 0.
\end{displaymath}
 If we let $K_i$ be the kernel of the map, $M_{i} \to M_{i-1}$, with $i > 0$ and $K_0 = A$, we get a short exact sequences,
\begin{displaymath}
 0 \to K_i \to M_i \to K_{i-1} \to 0,
\end{displaymath}
 providing maps $K_{i-1} \to K_{i}[1]$ in $\op{D}(\op{Mod-}A)$. The composition of morphisms,
 \begin{displaymath}
 A \to K_1[1] \to \cdots \to K_{l-1}[l-1] \to M[l],
 \end{displaymath} is the original element of $\op{Ext}_{A^e}^l(A,M)$. As the induced natural transformation, 
\begin{displaymath}
 \op{Hom}_A(\bullet,K_{i-1}) \to \op{Hom}_A(\bullet,K_{i}[1]),
\end{displaymath}
 vanishes on $\langle A^e \rangle_0$, Lemma 4.11 of \cite{Ro2} says that the natural transformation, 
\begin{displaymath}
 \op{Hom}_{A^e}(\bullet,A) \to \op{Hom}_{A^e}(\bullet,M[l]),
\end{displaymath}
resulting from the composition, vanishes on $\langle A^e \rangle_{l-1}$. If $l \geq d+1$, as $A \in \langle A^e \rangle_d$, we see that $\op{Ext}^l_{A^e}(A,\bullet)$ is identically zero. Thus, $\text{hd}(A) \leq d$.  In particular the Hochschild dimension is finite.
 
 On the other hand, by taking a projective resolution we see that $A \in \langle \overline{A^e} \rangle_{\text{hd}(A)}$.  Recall that Proposition $2.2.4$ of \cite{BV} states that, if $X$ and $Y$ are compact objects in a triangulated category and $X \in \langle \overline{Y} \rangle_s$, then $X \in \langle Y \rangle_s$. As $A$ and $A^e$ are perfect, they are compact.  It follows that $A \in \langle A^e \rangle_{\text{hd}(A)}$.  Hence $d \leq \text{hd}(A)$.
\end{proof}

\begin{lemma} \label{lem:Hoch bound on dimension}
 The generation time of $A$, in $\emph{D}_{\emph{perf}}(A)$, is bounded above by the Hochschild dimension of $A$. In particular, the Rouquier dimension of $\op{D}_{\op{perf}}(A)$ does not exceed the Hochschild dimension of $A$.
\end{lemma}

\begin{proof}
 The statement is vacuous if $\text{hd}(A) = \infty$ so we assume that $\text{hd}(A)$ is finite. Thus, one has $A \in \langle \overline{A^e} \rangle_{\text{hd}(A)}$. Taking $M \in \text{D}_{\text{perf}}(A)$, and applying the exact functor, $M \stackrel{\textbf{L}}{\otimes}_A \bullet: \op{D}(\op{Mod-}A^e) \to \op{D}(\op{Mod-}A)$, Lemma \ref{lem:functor invariant} tells us that,
\begin{displaymath}
 M \cong M \stackrel{\textbf{L}}{\otimes}_A A \in \langle \overline{M \stackrel{\textbf{L}}{\otimes}_A A^e} \rangle_{\text{hd}(A)} = \langle \overline{M \otimes_k A} \rangle_{\text{hd}(A)} = \langle \overline{A} \rangle_{\text{hd}(A)}.
\end{displaymath}
 The final equality holds as $M \otimes_k A$ is a, possibly infinite, coproduct of copies of $A$.  Applying Proposition $2.2.4$ of loc. cit., we conclude that $M$ lies in $\langle A \rangle_{\text{hd}(A)}$. Thus, $\tritime(A) \leq \text{hd}(A)$.
 
 As the Rouquier dimension of the minimum of generation times of strong generators of $\op{D}_{\op{perf}}(A)$, it is also bounded above by the Hochschild dimension of $A$.
\end{proof}

We will use the following lemma to compute the Hochschild dimension:

\begin{lemma}
 Assume that $A$ is a perfect $A^e$-module. The Hochschild dimension of $A$ is the maximal $i$ for which $\emph{Ext}^i_{A^e}(A,A^e)$ is nonzero.
\label{lem:Extest}
\end{lemma}

\begin{proof} We have seen that $A$ has finite Hochschild dimension. Take a resolution of $A$ by projective $A^e$-modules:
\begin{displaymath}
 0 \ra P_n \ra P_{n-1} \ra \cdots \ra P_0 \ra A \ra 0.
\end{displaymath}
Let $i_0$ be the maximal $i$ so that $\text{Ext}^i_{A^e}(A,A^e)$ is nonzero. It is clear that $i_0$ must be less than or equal to $n$. If $i_0$ is strictly less than $n$, then $\Ext^n_{A^e}(A,P)$ is zero for any projective module $P$. Thus, the map $P_n \ra P_{n-1}$ must split allowing us to shorten the projective resolution.
\end{proof}

H. Krause and D. Kussin, using a construction due to J. D. Christensen, prove the following (see Proposition 2.6 of \cite{KK}, the lower bound is Lemma 7.13 of \cite{Ro2}):

\begin{theorem} \label{CKK}
Let $A$ be a right-coherent $k$-algebra and view it as an object of $\emph{D}^{\emph{b}}(\emph{mod-}A)$. The generation time of $A$ is the global dimension of $A$.
\end{theorem}

\begin{remark}
 In \cite{Ro2}, Rouquier proved Theorem \ref{CKK} in the cases that $A$ is finite-dimensional over $k$ or $A$ is a commutative, local, and Noetherian $k$-algebra.
\end{remark}

In a special case of importance to us, we have equality of global and Hochschild dimensions. The following lemma is Proposition 7.4 of \cite{Ro2}:

\begin{lemma} \label{finite algebras}
If $A$ is a finite-dimensional algebra over a perfect field $k$, then the Hochschild dimension of $A$ equals the global dimension of $A$.
\end{lemma}

\begin{corollary} \label{cor:alldimensionsareequal}
 If $A$ is a finite-dimensional algebra over a perfect field, the Hochschild dimension, the global dimension, and the generation time of $A$ are equal.
\end{corollary}

For a variety (or an algebraic stack), we propose the following definition which is a weaker analogue of the Hochschild dimension. Recall that the K\"unneth (or exterior) product of an element $G \in \op{D}_{\op{qcoh}}(X)$ and an element $H \in \op{D}_{\op{qcoh}}(Y)$ is $p_1^*G \overset{\mathbf{L}}{\otimes}_{\mathcal O_{X \times Y}} p_2^*H \in \op{D}_{\op{qcoh}}(X \times Y)$, where $p_1: X \times Y \to X$ and $p_2: X \times Y \to Y$ are the projections. We denote $p_1^*G \overset{\mathbf{L}}{\otimes}_{\mathcal O_{X \times Y}} p_2^*H$ by $G \boxtimes H$. 

\begin{definition}
Let $X$ be a variety. The \textbf{diagonal dimension} of $X$, denoted $\idd(X)$, is the minimal $n$ such that the diagonal, $\mathcal{O}_{\Delta X}$, is in $\langle  G \boxtimes H \rangle_{n}$ for some $G \boxtimes H \in \emph{D}^{\emph{b}}_{\emph{coh}}(X \times X)$. It is set to $\infty$ if no such $n$ exists.
\end{definition}

The diagonal dimension has the following nice properties, the proofs of which, for the most part, are embedded in the next section:

\begin{lemma}
Let $X$ be a variety. One has:
\begin{enumerate}

\item $\idd(X \times Y) \leq \idd(X) + \idd(Y)$;

\item if $X$ is proper, then $\emph{dim}\ider \leq \idd(X)$;

\item if $X$ is smooth, then $\idd(X) \leq 2 \dim X$.

\end{enumerate}

\end{lemma}

Throughout this paper we obtain upper bounds on $\text{dim}\der$ by bounding $\dd(X)$, but, for the most part, we will simply state this bound either on $\text{dim}\der$ or on the generation time of the object being considered.

\subsection{Dimension for Deligne-Mumford stacks}

While stacks are not essential to the main arguments in this paper, they may provide a useful means for proving Conjecture \ref{conj:1}, see subsection \ref{subsec:weightedproj}. It is also natural to generalize Theorem \ref{thm:main} to stacks to obtain a greater class of examples. Consequently, in this subsection, we extend some of the basic results on Rouquier dimension to smooth and tame Deligne-Mumford stacks with quasi-projective coarse moduli spaces. All stacks are separated and finite-type over $k$.

Recall that a \textbf{Deligne-Mumford stack} over $k$, $\mathcal X$, is a stack in the \'etale topology on schemes over $k$ satisfying the following conditions:
\begin{enumerate}
 \item The diagonal, $\Delta_{\mathcal X}: \mathcal X \to \mathcal X \times_k \mathcal X$, is representable, quasi-compact, and separated.
 \item There exist a scheme, $U$, and an \'etale surjective morphism, $U \to \mathcal X$. 
\end{enumerate}
Given such an $\mathcal X$, the \textbf{inertia stack} of $\mathcal X$, $\mathcal{IX}$, is defined to be $\mathcal{IX} := \mathcal X \times_{\mathcal X \times_k \mathcal X}\mathcal X$. By \cite{KM97}, if $\mathcal{IX} \to \mathcal X$ is finite, $\mathcal X$ possesses a coarse moduli space $\pi: \mathcal X \to X$. If $\mathcal{IX} \to \mathcal X$ is finite and $\pi_*: \op{QCoh}(\mathcal X) \to \op{QCoh}(X)$ is exact, then $\mathcal X$ is called \textbf{tame} by \cite{AOV08}. 

\begin{lemma} \label{lem:lowerstacks}
 Let $\mathcal{X}$ be a tame Deligne-Mumford stack with a reduced and separated coarse moduli space. The Rouquier dimension of $\emph{D}^\emph{b}_{\icoh}(\mathcal{X})$ is at least the dimension of $\mathcal{X}$.
\end{lemma}

\begin{proof} 
 Let $\pi: \mathcal X \to X$ be the coarse moduli space. From the universal property of $X$, the natural map, $\mathcal O_X \to \pi_* \mathcal O_{\mathcal X}$, must be an isomorphism. From Proposition 4.5 of \cite{Alp08}, the counit map of adjunction, $\mathcal E \to \pi_* \pi^* \mathcal E$, is an isomorphism for any quasi-coherent sheaf, $\mathcal E$, on $X$. Consequently, $\pi_*: \op{D}^{\op{b}}_{\op{coh}}(\mathcal X) \to \op{D}^{\op{b}}_{\op{coh}}(X)$ is essentially surjective. This gives $\op{dim} \dbcoh{\mathcal X} \geq \op{dim} \dbcoh{X}$. If $X$ is a scheme, we can appeal to Rouquier's lower bound, Proposition 7.17 of \cite{Ro2}, to get $\op{dim} \dbcoh{X} \geq \op{dim} X$. Since the dimension of $\mathcal X$ is equal to the dimension of $X$, we would be done. However, in general, we can only assume that $X$ is an algebraic space. Thus, to finish the argument, we need to prove Rouquier's lower bound in the case of the reduced, separated, and finite-type algebraic space, $X$.

 From Proposition II.6.7 of \cite{Knu71}, there is dense open subspace of $X$ that is isomorphic to a scheme. Denote this subspace by $Y$. Let $G$ be a generator of $\op{D}^{\op{b}}_{\op{coh}}(X)$ and denote the inclusion by $j: Y \to X$. As in Proposition 7.17 of \cite{Ro2}, choose a closed point, $p$, of $Y$ such that $j^*G \otimes_{\mathcal O_Y} \mathcal O_{Y,p}$ is sum of shifts of free modules. Let $\mathcal O_p$ denote the pushforward of $k(p)$ under $\op{Spec}(k(p)) \to Y \to X$. If $\mathcal O_p$ lies in $\langle G \rangle_t$, then $k(p)$ lies in $\langle \mathcal{O}_{Y,p} \rangle_t$ and $t \geq \op{dim} \mathcal O_{Y,p}$ by Proposition 7.14 of loc. cit. The dimension of $\mathcal O_{Y,p}$ equals the dimension of $Y$, which equals the dimension of $X$. This gives Rouquier's lower bound for a reduced, separated, and finite-type algebraic space and finishes the argument.
\end{proof}

\begin{remark}
 If $\mathcal X$ is a tame Artin stack with coarse moduli space, $X$, one can prove $\dim \op{D}^{\op{b}}_{\op{coh}}(\mathcal X) \geq \dim X$ using essentially the same argument as in the proof of Lemma \ref{lem:lowerstacks}. 
\end{remark}

To get an upper bound on the Rouquier dimension of $\text{D}^\text{b}_{\coh}(\mathcal{X})$, in terms of the dimension of $\mathcal{X}$, we further restrict our class of stacks.

\begin{definition}
 Let $\mathcal{X}$ be a Deligne-Mumford stack with coarse moduli space, $\pi: \mathcal{X} \ra X$. A locally-free coherent sheaf, $\mathcal{E}$, on $\mathcal{X}$ is called a \textbf{generating sheaf} if for any quasi-coherent sheaf, $\mathcal{F}$, on $\mathcal{X}$, the natural morphism
\begin{displaymath}
 \pi^*(\pi_* \mathcal{H}om_{\mathcal{O}_{\mathcal{X}}}(\mathcal{E},\mathcal{F})) \otimes_{\mathcal{O}_{\mathcal{X}}} \mathcal{E} \ra \mathcal{F}
\end{displaymath}
is surjective.
\end{definition}

Following \cite{EHKV,OS,Kre}, we can give a useful construction of a generating sheaf. Assume that $\mathcal{X}$ is isomorphic to a global quotient stack, i.e.\ $\mathcal X \cong [Y/G]$ where $Y$ is a scheme and $G$ is a subgroup of $GL_n$ acting on $Y$. Take a $G$-representation $W$ which has an open subset $U$ where $G$ acts freely. At every geometric point of $\mathcal X$, the geometric stabilizer group acts faithfully on the fiber of the vector bundle $[(Y \times W)/G]$. Denote the associated locally-free coherent sheaf by $\mathcal{E}$. Then, $\bigoplus_{i=1}^r \mathcal{E}^{\otimes i}$ is a generating sheaf for $r$ large, see Section 5.2 of \cite{Kre} and Theorem 5.2 of \cite{OS}. 

This explicit construction of a generating sheaf lets us make a useful observation: since all the above procedures respect products, there is a generating sheaf on $\mathcal{X} \times \mathcal{X}$ which is an exterior product. Recall that an exterior product, $\mathcal F \boxtimes \mathcal G$, of sheaves, $\mathcal F$ and $\mathcal G$, is called K\"unneth-type. We can combine this observation with another from \cite{Kre}. Assume $[Y/G]$ has a quasi-projective coarse moduli space and let $\mathcal{L}$ be an ample line bundle on it. For any coherent sheaf, $\mathcal{F}$, on $\mathcal{X}$, there exists an $n_0$ so that the map,
\begin{displaymath}
 \text{Hom}_{\mathcal{X}}(\bigoplus_{i=1}^r \mathcal{E}^{\otimes i} \otimes_{\mathcal{O}_{\mathcal{X}}} \pi^* \mathcal{L}^{\otimes -n},\mathcal{F}) \otimes_k (\bigoplus_{i=1}^r \mathcal{E}^{\otimes i} \otimes_{\mathcal{O}_{\mathcal{X}}} \pi^*\mathcal{L}^{\otimes -n}) \ra \mathcal{F},
\end{displaymath}
is surjective for $n \geq n_0$. In particular, on $\mathcal{X} \times \mathcal{X}$, we can use $\left(\bigoplus_{i=1}^r \mathcal{E}^{\otimes i}\right) \boxtimes \left(\bigoplus_{i=1}^r \mathcal{E}^{\otimes i}\right)$ for our generating sheaf and $\mathcal{L} \boxtimes \mathcal{L}$ as the ample line bundle on the coarse moduli space. Thus, for any quasi-coherent sheaf $\mathcal{F}$ on $\mathcal{X} \times \mathcal{X}$, there is a locally-free K\"unneth-type sheaf surjecting onto $\mathcal{F}$.

Consequently, $\mathcal X$ has enough locally-free sheaves, see \cite{Tot} for a thorough discussion of geometric implications of this. It is simple to check, as in \cite{Nee}, that any locally-free coherent sheaf is a compact object of $\op{D}_{\op{qcoh}}(\mathcal X)$ and, if we have enough locally-free sheaves, $\op{D}_{\op{qcoh}}(\mathcal X)$ is compactly-generated.

\begin{lemma} \label{lem:upperstacks}
 Let $\mathcal{X}$ be a smooth and tame Deligne-Mumford stack with quasi-projective coarse moduli space. The Rouquier dimension of $\emph{D}^\emph{b}_{\emph{\coh}}(\mathcal{X})$ is finite. If $\mathcal X$ is proper, then the Rouquier dimension is bounded by twice the dimension of $\mathcal{X}$.
\end{lemma}

\begin{proof}
 Let $n$ denote the dimension of $\mathcal X$. By Theorem $4.4$ of \cite{Kre}, $\mathcal{X}$ is automatically a global quotient stack.
 
 Since we have enough locally-free sheaves on $\mathcal X \times \mathcal X$, we can apply Serre's theorem on regularity to an \'etale cover of $\mathcal X \times \mathcal X$ to conclude that $\mathcal{E}xt^p_{\mathcal X \times \mathcal X}(\mathcal F_1,\mathcal F_2)$ vanishes for $p > 2n$ and for any coherent sheaves, $\mathcal F_1$ and $\mathcal F_2$. Using Grothendieck's vanishing theorem, Theorem 1.10 \cite{Kri}, we see that $\op{H}^q(\mathcal X \times \mathcal X,\mathcal{E}xt^p_{\mathcal X \times \mathcal X}(\mathcal F_1,\mathcal F_2))$ vanishes for $p+q > \dim 4n$. Furthermore, we have a spectral sequence, associated to the identity $\Gamma \circ \mathcal{H}om_{\mathcal X \times \mathcal X}(\mathcal{F}_1,\bullet) \cong \op{Hom}_{\mathcal X \times \mathcal X}(\mathcal{F}_1,\bullet)$, converging to $\op{Ext}^{p+q}_{\mathcal X \times \mathcal X}(\mathcal F_1, \mathcal F_2)$ whose $E_2$-page is $\op{H}^q(\mathcal X \times \mathcal X,\mathcal{E}xt^p_{\mathcal X \times \mathcal X}(\mathcal F_1,\mathcal F_2))$.  It follows that $\op{Ext}^{r}_{\mathcal X \times \mathcal X}(\mathcal F_1, \mathcal F_2) =0$ for $r > 4n$.
 
 Now take the structure sheaf of the diagonal, $\mathcal{O}_{\Delta \mathcal{X}}$, and resolve it by finite rank locally-free K\"unneth-type sheaves:
\begin{displaymath}
 \cdots \ra \mathcal{H}_m \boxtimes \mathcal{G}_m \ra \cdots \ra \mathcal{H}_0 \boxtimes \mathcal{G}_0 \ra \mathcal{O}_{\Delta \mathcal{X}} \ra 0.
\end{displaymath}
 Let $\mathcal{K}$ be the kernel of the map $\mathcal{H}_{4n} \boxtimes \mathcal{G}_{4n} \to \mathcal{H}_{4n-1} \boxtimes \mathcal{G}_{4n-1}$.
 Taking the brutal truncation of the resolution, $\sigma^{\geq 4n} \mathcal{H}_\bullet \boxtimes \mathcal{G}_\bullet $, we have an exact triangle:
 \begin{displaymath}
 \mathcal{K}[4n] \to \sigma^{\geq 4n} \mathcal{H}_\bullet \boxtimes \mathcal{G}_\bullet  \to \mathcal{O}_{\Delta \mathcal{X}} \to  \mathcal K[4n+1].
 \end{displaymath}
 As the final map is an element of $\op{Ext}^{4n+1}_{\mathcal X \times \mathcal X}(\mathcal{O}_{\Delta \mathcal{X}}, \mathcal K)$, it must vanish and  the triangle splits. Hence, $\mathcal{O}_{\Delta \mathcal{X}}$ is a summand of the brutal truncation, $\sigma^{\geq 4n} \mathcal{H}_\bullet \boxtimes \mathcal{G}_\bullet$, and therefore lies in $\left< \bigoplus_{i=0}^{4n} \mathcal{H}_{i} \boxtimes \mathcal{G}_{i} \right>_{4n}$.

 Let $\mathcal{O}_{\Delta \mathcal{X}}$ lie in $\left< \bigoplus_{i=0}^{t} \mathcal{H}_{i} \boxtimes \mathcal{G}_{i} \right>_{t}$ for some $t$.
 Denote by $p_1,p_2$ the projections from $\mathcal{X} \times \mathcal{X}$ to the first, second factor, respectively. Take a bounded complex of coherent sheaves, $\mathcal{F}$, on $\mathcal{X}$. Applying $\Phi_{\bullet}(\mathcal{F}):=\textbf{R}p_{2*}(\bullet \stackrel{\textbf{L}}{\otimes}_{\mathcal{O}_{\mathcal{X}}} \textbf{L}p_1^*(\mathcal{F})): \op{D}_{\op{qcoh}}(\mathcal X \times \mathcal X) \to \op{D}_{\op{qcoh}}(\mathcal X)$ to the previous statement, we get 
\begin{displaymath}
 \mathcal{F} \cong \Phi_{\mathcal{O}_{\Delta \mathcal{X}}}(\mathcal{F}) \in \left< \bigoplus_{i=0}^{t} \text{H}^*(\mathcal{X},\mathcal{H}_{i} \otimes_{\mathcal{O}_{\mathcal{X}}}\mathcal{F}) \otimes_k \mathcal{G}_{i} \right>_{t} \subseteq \left< \overline{\bigoplus_{i=0}^{t} \mathcal{G}_{i}} \right>_{t}.
\end{displaymath}
 Applying Proposition $2.2.4$ of \cite{BV}, we see that $\mathcal F$ lies in $\left< \oplus_{i=0}^{t} \mathcal{G}_{i} \right>_{t}$.

 Instead of the spectral sequence argument to get Ext-vanishing, if $\mathcal X$ is proper, then we can use Serre duality, see Lemma \ref{lem:Serre}, to conclude that $\op{Ext}^{2n+1}_{\mathcal X \times \mathcal X}(\mathcal F_1, \mathcal F_2) =0$. Running the same proof gives the bound in this case.
\end{proof}

We have the contravariant functor, $(\bullet)^{\vee}: \op{D}_{\op{qcoh}}(\mathcal X) \to \op{D}_{\op{qcoh}}(\mathcal X)$, defined by $\mathcal F^{\vee} := \mathbf{R}\mathcal{H}om_{\mathcal X}(\mathcal F,\mathcal O_{\mathcal X})$. $(\bullet)^{\vee}$ restricts to an anti-auto-equivalence of $\op{D}_{\op{perf}}(\mathcal X)$.

\begin{lemma} \label{lem:diaggenstack}
 If $G$ is a generator of $\isder$, with $\mathcal X$ a smooth and tame Deligne-Mumford stack possessing a quasi-projective coarse moduli space, then $G \boxtimes G^{\vee}$ is a strong generator for $\isdder$.
\end{lemma}

\begin{proof}
 By Lemma \ref{lem:autostrong}, $G$ is a strong generator and it is sufficient to prove that $G \boxtimes G^{\vee}$ is a generator of $\op{D}^{\op{b}}_{\op{coh}}(\mathcal X \times \mathcal X)$. As $\mathcal X$ is smooth and possesses enough locally-free coherent sheaves, we can use the arguments in the proof of Lemma \ref{lem:upperstacks} to show any coherent sheaf is a perfect complex. Consequently,  $\op{D}^{\op{b}}_{\op{coh}}(\mathcal X)$ is equivalent to $\op{D}_{\op{perf}}(\mathcal X)$. As generation time is invariant under auto-equivalences, $G^{\vee}$ is also strong generator. By taking bounded locally-free resolutions, the collection of sheaves of the form,
 \begin{displaymath}
 \left(\bigoplus_{i=0}^r \mathcal E^{\otimes i} \otimes_{\mathcal O_{\mathcal X}} \mathcal \pi^*\mathcal{L}^{-\otimes n}\right) \boxtimes \left(\bigoplus_{i=0}^r \mathcal E^{\otimes i} \otimes_{\mathcal O_{\mathcal X}} \pi^*\mathcal L^{-\otimes n}\right),
 \end{displaymath}
  with $n \geq 0$ and $\mathcal L$ an ample line bundle on $X$, generates $\op{D}^{\op{b}}_{\op{coh}}(\mathcal X \times \mathcal X)$. Thus, we just need to generate these vector bundles using $ G \boxtimes G^{\vee}$.
 
 We know that,
\begin{displaymath}
 \bigoplus_{i=0}^r \mathcal E^{\otimes i} \otimes_{\mathcal O_{\mathcal X}} \mathcal \pi^*\mathcal L^{-\otimes n} \in \langle G \rangle_t
 \end{displaymath}
for some $t$. By Lemma \ref{lem:functor invariant}, we have
\begin{displaymath}
 \left(\bigoplus_{i=0}^r \mathcal E^{\otimes i} \otimes_{\mathcal O_{\mathcal X}} \mathcal \pi^*\mathcal L^{-\otimes n}\right) \boxtimes G^{\vee} \in \langle G \boxtimes G^{\vee} \rangle_t.
\end{displaymath}
Similarly, we have,
\begin{displaymath}
 \bigoplus_{i=0}^r \mathcal E^{\otimes i} \otimes_{\mathcal O_{\mathcal X}} \mathcal \pi^*\mathcal L^{-\otimes n} \in \langle G^{\vee} \rangle_s,
\end{displaymath}
for some $s$. Using Lemma \ref{lem:functor invariant} again, we have,
\begin{displaymath} 
 \left( \bigoplus_{i=0}^r \mathcal E^{\otimes i} \otimes_{\mathcal O_{\mathcal X}} \mathcal \pi^*\mathcal L^{-\otimes n} \right) \boxtimes \left( \bigoplus_{i=0}^r \mathcal E^{\otimes i} \otimes_{\mathcal O_{\mathcal X}} \mathcal \pi^*\mathcal L^{-\otimes n} \right) \in \left\langle \left(\bigoplus_{i=0}^r \mathcal E^{\otimes i} \otimes_{\mathcal O_{\mathcal X}} \mathcal \pi^*\mathcal L^{-\otimes n}\right) \boxtimes G^{\vee} \right\rangle_s.
\end{displaymath}
and
\begin{displaymath}
 \left\langle \left(\bigoplus_{i=0}^r \mathcal E^{\otimes i} \otimes_{\mathcal O_{\mathcal X}} \mathcal \pi^*\mathcal L^{-\otimes n}\right) \boxtimes G^{\vee} \right\rangle_s \subset \underbrace{\langle G \boxtimes G^{\vee} \rangle_t \diamond \cdots \diamond \langle G \boxtimes G^{\vee} \rangle_t}_s = \langle G \boxtimes G^{\vee} \rangle_{s+t+1}
\end{displaymath}
by Lemma \ref{lem:associative}. Thus,
\begin{displaymath}
 \left( \bigoplus_{i=0}^r \mathcal E^{\otimes i} \otimes_{\mathcal O_{\mathcal X}} \mathcal \pi^*\mathcal L^{-\otimes n} \right) \boxtimes \left( \bigoplus_{i=0}^r \mathcal E^{\otimes i} \otimes_{\mathcal O_{\mathcal X}} \mathcal \pi^*\mathcal L^{-\otimes n} \right) \in \langle G \boxtimes G^\vee \rangle_{s+t+1}
\end{displaymath}
and we can conclude that $G \boxtimes G^{\vee}$ generates.
\end{proof}

In the case that $X$ is a smooth and quasi-projective variety, the following lemma is corollary of Lemma \ref{lem:diaggenstack}. However, something stronger can be said for a general smooth variety.

\begin{lemma} \label{lem:diaggen}
 If $G$ is a generator of $\ider$, with $X$ a smooth variety, then $G \boxtimes G^{\vee}$ is a strong generator for $\idder$.
\end{lemma}

\begin{proof}
 As already noted above, $G$ and $G^{\vee}$ are strong generators for $\op{D}^{\op{b}}_{\op{coh}}(X)$. Any generator, $G'$, of $\op{D}^{\op{b}}_{\op{coh}}(X)$ is a compact generator of $\op{D}_{\op{qcoh}}(X)$, in the sense of \cite{BV}. Were $G'$ to have nontrivial right orthogonal in $\op{D}_{\op{qcoh}}(X)$, then so would $\op{D}^{\op{b}}_{\op{coh}}(X)$. We can apply Proposition 3.1.4 of loc. cit. and note that $G \boxtimes G^{\vee}$ is a compact generator of $\op{D}_{\op{qcoh}}(X \times X)$. Proposition 2.2.4 of loc. cit. implies that $G \boxtimes G^{\vee}$ is then a generator of $\op{D}^{\op{b}}_{\op{coh}}(X \times X)$ and Theorem 3.1.4 of loc. cit. states that $\op{D}^{\op{b}}_{\op{coh}}(X \times X)$ has a strong generator. Lemma \ref{lem:autostrong} states that $G \boxtimes G^{\vee}$ must be a strong generator as well.
\end{proof}

% The following lemma is demonstrates that dimension only decreases under algebraic field extensions. It appears as Proposition 5.1.2 of the first version of \cite{KMV08}. It has since been removed in later versions so we feel it is prudent to recall it.
% 
% \begin{lemma}
%  Let $X$ be a variety. If $K$ is an algebraic field extension of $k$, then the Rouquier dimension of $X$ is bounded above by the Rouquier dimension of $X_K=X \times_{\op{Spec} k} \op{Spec} K$.
% \end{lemma}
% 
% \begin{proof}
%  The proof is exactly as in \cite{KMV08}. Let $G$ be a generator of $\dbcoh{X_K}$. There a subfield of $K$, call it $k'$, that is a finite extension of $k$, and there is an object $G'$ of $\dbcoh{X_{k'}}$ so that $\rho^*G' = G$, where $\rho: X_{K} \to X_{k'}$ is the map induced by base change. Let $\pi: X_K \to X$ be the map induced by base change from $k$ to $K$. $\pi_*: \op{D}^{\op{b}}_{\op{qcoh}}(X_K) \to \op{D}^{\op{b}}_{\op{qcoh}}(X)$ is dense by the projection formula. So, 
% \end{proof}

\subsection{Tilting objects and Serre functors}

\begin{definition} \label{defn:tilt}
 Let $\mathcal{T}$ be a $k$-linear triangulated category.  An object, $T$, of $\mathcal{T}$ is called a \textbf{tilting object} if the following two conditions hold:
\begin{enumerate}
\item $\emph{Hom}_{\mathcal{T}}(T, T[i]) = 0 \emph{ for all } i \not = 0$;

\item $T$ is a generator for $\mathcal{T}$.
\end{enumerate}
\end{definition}

Our tilting objects will mainly reside in the bounded derived category of coherent sheaves on a smooth variety, $X$. 

\begin{proposition} \label{tilt}
 Let $T$ be an object of $\ider$, where $X$ is smooth, and set $A:= \op{End}_X(T)$. If $T$ satisfies condition $(1)$ of Definition \ref{defn:tilt}, then there is a full, faithful, and exact functor, $\bullet \overset{\mathbf{L}}{\otimes}_A T: \op{D}_{\op{perf}}(A) \to \op{D}^{\op{b}}_{\op{coh}}(X)$. If $T$ is a tilting object, $\bullet \overset{\mathbf{L}}{\otimes}_A T$ is an equivalence.
\end{proposition}

\begin{proof}
 The proof of Theorem 9.2 of \cite{Kel} produces an exact functor, $\bullet \overset{\mathbf{L}}{\otimes}_A T: \op{D}(\op{Mod-}A) \to \op{D}_{\op{qcoh}}(X)$, that commutes with coproducts. Lemma 4.2.b of loc. cit. shows that it is full and faithful. It is clear that the restriction of $\bullet \overset{\mathbf{L}}{\otimes}_A T$ to $\op{D}_{\op{perf}}(A)$ has image in $\op{D}^{\op{b}}_{\op{coh}}(X)$. 

 If $T$ is a tilting object, then Theorem 9.2 of loc. cit. states that $\bullet \overset{\mathbf{L}}{\otimes}_A T: \op{D}(\op{Mod-}A) \to \op{D}_{\op{qcoh}}(X)$ is an exact equivalence. As $\op{D}_{\op{perf}}(A)$ and $\op{D}_{\op{perf}}(X)$ are the subcategories of compact objects of $\op{D}(\op{Mod-}A)$ and $\op{D}_{\op{qcoh}}(X)$, respectively, $\op{D}_{\op{perf}}(A)$ and $\op{D}_{\op{perf}}(X)$ are equivalent. $\op{D}^{\op{b}}_{\op{coh}}(X)$ is equivalent to $\op{D}_{\op{perf}}(X)$ as $X$ is smooth.
\end{proof}

% \begin{proof}
% Consider the functor $\Psi:= \bullet \overset{\mathbf{L}}{\otimes}_A T$.  We have $\Psi(A) = A \overset{\mathbf{L}}{\otimes}_A T \cong T$ and this identification induces an isomorphism, $A \cong \text{Hom}_A(A,A) \cong \text{Hom}_X(\Psi(A),\Psi(A)) \cong \text{Hom}_X(T,T) = A$. Therefore, $\Psi$ is full and faithful on the full subcategory consisting of the object $A$.  Since $\Psi$ commutes with shifts, taking direct sums, taking direct summands, and sends triangles to triangles, $\Psi$ is full and faithful on $\langle A \rangle_{\infty}$. Now, as $\Psi$ is full, faithful, and exact, the essential image of $\Psi$ is triangulated.  Since $\Psi(A)\cong T$, $\Psi$ essentially surjects onto the smallest thick triangulated subcategory of $\text{D}^\text{b}_{\text{coh}}(X)$ containing $T$ which by assumption is all of $\text{D}^\text{b}_{\text{coh}}(X)$.
% \end{proof}

\begin{definition}
A $k$-linear exact autoequivalence, $S$, of $\mathcal{T}$, is called a \textbf{Serre functor} if
for any pair of objects, $X$ and $Y$, of $\mathcal{T}$, there exists an isomorphism of vector spaces,
$$\emph{Hom}_{\mathcal T}(Y,X)^* \cong \emph{Hom}_{\mathcal T}(X, S(Y)),$$
which is natural in $X$ and $Y$.
\end{definition}
A Serre functor, if it exists, is determined uniquely up to natural isomorphism. If $F: \mathcal{T} \ra \mathcal{S}$ is an exact equivalence of triangulated categories possessing Serre functors, then $F$ commutes with those Serre functors \cite{BK}.

There are two main examples of Serre functors. For the first, let $A$ be a $k$-algebra that is finite-dimensional. $A^* = \op{Hom}_k(A,k)$ determines a functor, $A^* \overset{\mathbf{L}}{\otimes}_A \bullet : \op{D}_{\op{perf}}(A) \to \op{D}^{\op{b}}(\op{mod-}A)$. If $A$ has finite global dimension, then $A^* \overset{\mathbf{L}}{\otimes}_A \bullet$ is the Serre functor on $\op{D}_{\op{perf}}(A) \cong \op{D}^{\op{b}}(\op{mod-}A)$, see Example 3.2.3 of \cite{BK} and Theorem 3.6 of \cite{Hap}. The second is Example 3.2.2 of \cite{BK}: the Serre functor on a smooth and proper variety, $X$, is $\bullet \otimes_{\mathcal O_X} \omega_X[\dim X]: \dbcoh{X} \to \dbcoh{X} $

We have a quasi-isomorphism, $\epsilon_X^!\mathcal O_{\op{Spec} k} \cong \omega_X[\dim X]$, where $\epsilon_X^!: \op{D}(\op{Vect} k) \to \op{D}_{\op{qcoh}}(X)$ is the right adjoint to the derived pushforward along the structure map, $\epsilon_X: X \to \op{Spec} k$. We wish to clarify the relation for a class of stacks.

As for a variety, the dualizing complex of a stack, $\mathcal X$, over $k$ is $\epsilon_{\mathcal X}^! \mathcal O_{\op{Spec} k}$ where $\epsilon_{\mathcal X}^!: \op{D}(\op{Vect} k) \to \op{D}_{\op{qcoh}}(\mathcal X)$ is the right adjoint of derived pushforward along the structure map, $\epsilon_{\mathcal X}: \mathcal X \to \op{Spec} k$. Of course, the dualizing complex may not exist so one must take care.

\begin{lemma}\label{lem:Serre}
 Let $\mathcal{X}$ be a smooth, proper, and tame Deligne-Mumford stack with a projective coarse moduli space. The Serre function on $\op{D}^{\op{b}}_{\op{coh}}(\mathcal X)$ is the functor, $\bullet \otimes_{\mathcal O_{\mathcal X}} \epsilon_{\mathcal X}^! \mathcal O_{\op{Spec} k}: \op{D}^{\op{b}}_{\op{coh}}(\mathcal X) \to \op{D}^{\op{b}}_{\op{coh}}(\mathcal X)$. If $\mathcal X$ is connected, $\epsilon_{\mathcal X}^! \mathcal O_{\op{Spec} k}$ is quasi-isomorphic to a line bundle shifted by $\dim \mathcal X$.
\end{lemma}

\begin{proof}
 By Lemma \ref{lem:upperstacks}, $\dbcoh{\mathcal X}$ has a strong generator. Theorem 1.3 of \cite{BV} says that $\dbcoh{\mathcal X}$ is saturated, and, therefore, $\dbcoh{\mathcal X}$ must possess a Serre functor. 

 Saturation also implies that $\epsilon_{\mathcal X}^!$ exists, see \cite{BK}. For any pair of perfect complexes, $\mathcal E$ and $\mathcal F$, we have a sequence of natural isomorphisms:
\begin{align*}
 \op{Hom}_{\mathcal X}(\mathcal E,\mathcal F \overset{\mathbf{L}}{\otimes}_{\mathcal O_{\mathcal X}}\epsilon_{\mathcal X}^! \mathcal O_{\op{Spec} k} ) & \cong \op{Hom}_{\mathcal X}(\mathbf{R}\mathcal{H}om_{\mathcal X}(\mathcal F,\mathcal E),\epsilon^!_{\mathcal X} \mathcal O_{\op{Spec} k}) \\ & \cong \op{Hom}_k(\mathbf{R}\epsilon_{\mathcal X*}(\mathbf{R}\mathcal{H}om_{\mathcal X}(\mathcal F,\mathcal E)),\mathcal{O}_{\op{Spec} k}) \\ & = \op{Hom}_{\mathcal X}(\mathcal F,\mathcal E)^*.
\end{align*}
 Since $\mathcal X$ is smooth, every object of $\dbcoh{\mathcal X}$ is perfect. Corollary 4.24 of \cite{Ro2} states that $\bullet \overset{\mathbf{L}}\otimes_{\mathcal O_{\mathcal X}} \epsilon_{\mathcal X}^! \mathcal O_{\op{Spec} k}$ actually has essential image in $\dbcoh{\mathcal X}$ and, indeed, is the Serre functor on $\op{D}^{\op{b}}_{\op{coh}}(\mathcal X)$.

 Let $Y$ be a Noetherian scheme. Given a perfect complex, $\mathcal E$, $\bullet \overset{\mathbf{L}}{\otimes}_{\mathcal O_Y} \mathcal E: \op{D}_{\op{perf}}(Y) \to \op{D}_{\op{perf}}(Y)$ is an autoequivalence if and only if, on each connected component of $Y$, there is an $i \in \Z$ so that the morphisms, $\mathcal E \to \tau_{\geq i}\mathcal E \leftarrow \mathcal{H}^i(\mathcal E)[-i]$, are quasi-isomorphisms and $\mathcal{H}^i(\mathcal E)$ is locally-free of rank one, see the proof of Lemma 6.6 of \cite{Bal}. Here, $\tau_{\geq i}$ is the gentle truncation in degrees $\geq i$, i.e. the truncation to degrees $\geq i$ so that there exists a map, $\mathcal E \to \tau_{\geq i} \mathcal E$, that is an isomorphism on the sheaf cohomology in degrees $\geq i$.

 As $\bullet \overset{\mathbf{L}}\otimes_{\mathcal O_{\mathcal X}} \epsilon_{\mathcal X}^! \mathcal O_{\op{Spec} k}$ is the Serre functor, it must be an auto-equivalence. Its left and right adjoint, $\bullet \overset{\mathbf{L}}\otimes_{\mathcal O_{\mathcal X}} \left( \epsilon_{\mathcal X}^! \mathcal O_{\op{Spec} k}\right)^{\vee}$, is then its inverse. If we take any map from a scheme, $Y$, to $\mathcal X$, tensoring by the pullback of $\epsilon_{\mathcal X}^! \mathcal O_{\op{Spec} k}$ is, therefore, an auto-equivalence of $\op{D}_{\op{perf}}(Y)$. Applying the result of the previous paragraph to each $Y$, we see that $\epsilon_{\mathcal X}^! \mathcal O_{\op{Spec} k}$ is quasi-isomorphic on each connected component of $\mathcal X$ to a fixed shift of a line bundle. Note that we now know we do not need to derive $\bullet \otimes_{\mathcal O_{\mathcal X}} \epsilon_{\mathcal X}^! \mathcal O_{\op{Spec} k}$ to pass it to the derived category.

 If $\mathcal X$ is connected, $\epsilon_{\mathcal X}^! \mathcal O_{\op{Spec} k}$ is quasi-isomorphic to $\omega_{\mathcal X}[d]$ for some line bundle, $\omega_{\mathcal X}$, on $\mathcal X$ and some $d \in \Z$. By tameness, the unit of adjunction, $\op{Id} \to \pi_*\pi^*$, is an isomorphism in $\op{Qcoh}(X)$, see the proof of Lemma \ref{lem:lowerstacks}. Therefore, the adjunction, $\op{Id} \to \pi_*\mathbf{L}\pi^*$, is also an isomorphism. Thus,
\begin{displaymath}
 \op{Hom}_{\mathcal X}(\mathbf{L}\pi^* \mathcal E,\mathbf{L}\pi^* \mathcal F) \cong \op{Hom}_{X}(\mathcal E,\pi_* \mathbf{L}\pi^* \mathcal F) \cong \op{Hom}_{X}(\mathcal E,\mathcal F).
\end{displaymath}
 Let $p$ be closed point of $X$ with $\mathcal O_{X,p}$ regular and $\pi_*\omega_{\mathcal X}\otimes_{\mathcal O_X} \mathcal O_{X,p} \cong \mathcal O_{X,p}^{\oplus r}$. Using Serre duality for $\mathcal X$,
\begin{displaymath}
 \op{Hom}_X(\mathcal O_X,\mathcal O_p)^* \cong \op{Hom}_{\mathcal X}(\mathcal O_{\mathcal X},\mathbf{L}\pi^*\mathcal O_p)^* \cong \op{Hom}_{\mathcal X}(\mathbf{L}\pi^* \mathcal O_p, \omega_{\mathcal X}[d]) \cong \op{Hom}_X(\mathcal O_p, \pi_*\omega_{\mathcal X}[d]).
\end{displaymath}
 The natural map, $\op{Hom}_X(\mathcal O_p,\pi_*\omega_{\mathcal X}[d]) \to \op{Hom}_{\mathcal O_{X,p}}(\mathcal O_p,\pi_*\omega_{\mathcal X} \otimes_{\mathcal O_X}\mathcal O_{X,p}[d])$, is an isomorphism as $\mathcal O_p$ is supported at $p$. Thus,
\begin{displaymath}
 \op{H}^0(X,\mathcal O_p)^* \cong \op{Ext}^d_{\mathcal O_{X,p}}(\mathcal O_p, \mathcal O_{X,p}^{\oplus r})
\end{displaymath}
 and we see that $d=\dim X = \dim \mathcal X$ (and $r=1$). 
\end{proof}

We record a definition of the canonical bundle.

\begin{definition} \label{def:canonical}
 If $\mathcal X$ is a smooth, proper, tame, and connected Deligne-Mumford stack with projective coarse moduli space, $\omega_{\mathcal X}$ will denote the line bundle, $\mathcal {H}^{-\dim \mathcal X}(\epsilon_{\mathcal X}^! \mathcal O_{\op{Spec} k})$, as in Lemma \ref{lem:Serre},
\end{definition} 

\begin{remark}\label{rmk}
 The definition above is not standard. If $k$ is algebraically-closed, Theorem 2.22 of \cite{Nir08} states that this definition coincides with the usual definition. If $k$ is not algebrically-closed, then presumably one could use base-change and the naturality of the Serre functor to check that two definitions coincide. However, we feel it is beyond the scope of this paper to fill this gap. So, we content ourselves with algebraically-closed fields for the majority of our examples involving stacks.
 
 Of course, if $\mathcal X$ is a variety, Definition \ref{def:canonical} agrees with the usual definition, regardless of the field.
\end{remark}

\section{Generation time for tilting objects}

We begin with a statement for a general class of triangulated categories. See \cite{Kel} for the definition of an algebraic triangulated category.

\begin{proposition} \label{prop:fdHochbound}
 Let $\mathcal{T}$ be a $k$-linear algebraic triangulated category with finite dimensional morphism spaces. Assume that $\mathcal{T}$ possesses a tilting object, $T$, and that $A:=\emph{End}_{\mathcal{T}}(T)$ lies in $\emph{D}_{\emph{perf}}(A^e)$. Let $S$ be the Serre functor for $\mathcal{T}$ and $j_0$ be the largest $j$ for which $\emph{Hom}_{\mathcal{T}}(T,S^{-1}(T)[j])$ is nonzero. The Hochschild dimension of $A$ is equal to $j_0$.  
\end{proposition}

\begin{proof}
 From our assumption that $\mathcal{T}$ is algebraic, there is an exact equivalence, $\bullet \overset{\mathbf{L}}{\otimes}_A T: \text{D}_{\perf}(A) \to \mathcal T$, see \cite{Kel}. Since $\mathcal T$ has finite-dimensional morphism spaces, $A$ is a finite-dimensional algebra. By Theorem $7.26$ of \cite{Ro2}, $A$ has finite global dimension and $\text{D}^\text{b}(\text{mod-}A)$ is equivalent to $\text{D}_{\perf}(A)$. As noted above, $\text{D}^\text{b}(\text{mod-}A)$ has a Serre functor. One can check that its inverse is $\bullet \overset{\mathbf{L}}{\otimes}_A \textbf{R}\text{Hom}_{A^e}(A,A^e)$, see \cite{Ginz}. By naturality,
\begin{displaymath}
 \text{Hom}_{\mathcal{T}}(T,S^{-1}(T)[j]) \cong \text{Hom}_{A}(A,\textbf{R}\text{Hom}_{A^e}(A,A^e)[j]) \cong \text{Ext}^j_{A^e}(A,A^e).
\end{displaymath}
Applying Lemma \ref{lem:Extest}, we get the result.
\end{proof}

We can use this to get bounds for stacks.

\begin{theorem} \label{thm:stackbound}
 Let $\mathcal{X}$ be a smooth, proper, tame, and connected Deligne-Mumford stack with a projective coarse moduli space. Suppose that $T$ is a tilting object in $\emph{D}^{\emph{b}}_{\emph{coh}}(\mathcal{X})$ and $i_0$ is the largest $i$ for which $\emph{Hom}_{\mathcal{X}}(T,T \otimes_{\mathcal O_{\mathcal X}} \omega_{\mathcal{X}}^{\vee}[i])$ is nonzero. The Hochschild dimension of $A:=\emph{End}_{\mathcal{X}}(T)$ is equal to $\emph{dim}(\mathcal X) + i_0$. Consequently, the generation time of $T$ is bounded above by $\dim \mathcal{X} + i_0$. If $k$ is perfect, then the generation time of $T$ is equal to $\dim \mathcal{X} + i_0$, in particular $\op{dim} \mathcal T \leq \dim \mathcal X = i_0$.
\end{theorem}

\begin{proof}
 Lemma \ref{lem:Serre} shows that $\bullet \otimes_{\mathcal{O}_{\mathcal X}} \omega_{\mathcal X}[\dim \mathcal X]$ is the Serre functor on $\dbcoh{\mathcal X}$. The theorem is then a consequence of Proposition \ref{prop:fdHochbound}.  If $\mathcal X$ is over a perfect field, Corollary \ref{cor:alldimensionsareequal} says that $\tritime(T) = \dim \mathcal X + i_0$.
\end{proof}

The following corollary can be used to show that some singular varieties have Rouquier dimensions equal to their Krull dimensions, e.g. weighted projective spaces.

\begin{corollary}
 Let $\mathcal{X}$ be a smooth, proper, tame, and connected Deligne-Mumford stack with a projective coarse moduli space, $\pi: \mathcal X \to X$. Suppose that $T$ is a tilting object in $\emph{D}^{\emph{b}}_{\emph{coh}}(\mathcal{X})$ and $i_0$ is the largest $i$ for which $\emph{Hom}_{\mathcal{X}}(T,T \otimes_{\mathcal O_{\mathcal X}} \omega_{\mathcal{X}}^{\vee}[i])$ is nonzero. The Rouquier dimension of $\emph{D}^{\emph{b}}_{\emph{coh}}(X)$ is bounded above by $\dim X + i_0$.
\end{corollary}

\begin{proof}
 In the proof of Lemma \ref{lem:lowerstacks}, we saw that $\pi_*: \op{D}^{\op{b}}_{\op{coh}}(\mathcal{X}) \to \op{D}^{\op{b}}_{\op{coh}}(X)$ is essentially surjective. Thus, by Lemmas \ref{density lemma} and \ref{lem:Hoch bound on dimension}, $\dim \op{D}^{\op{b}}_{\op{coh}}(X) \leq \op{D}^{\op{b}}_{\op{coh}}(\mathcal X) \leq \dim \mathcal X + i_0 = \dim X + i_0$.
\end{proof}

If we restrict ourselves to varieties, a stronger statement is possible. 

\begin{theorem} \label{main theorem}
Suppose $X$ is a smooth variety and $T$ is a tilting object in $\ider$.  Let $i_0$ be the largest $i$ for which $\emph{Hom}_X(T,T \otimes_{\mathcal O_X} \omega_X^{\vee}[i])$ is nonzero. The Hochschild dimension of $\emph{End}_X(T)$ is equal to $\emph{dim}(X) + i_0$. Consequently, the generation time of $T$ is bounded above by $\dim X + i_0$. If $X$ is proper over a perfect field, then the generation time of $T$ is equal to $\dim X + i_0$.
\end{theorem}

\begin{proof}
 Write $A$ as shorthand for $\text{End}_X(T)$. Then $A^e$ is isomorphic to $\text{End}_{X \times X}(T \boxtimes T^{\vee})$. By Lemma \ref{lem:diaggen}, $T \boxtimes T^{\vee}$ is a strong generator. Thus, $T \boxtimes T^{\vee}$ is a tilting object. By Proposition~\ref{tilt}, this yields an equivalence of categories between $\text{D}^{\text{b}}_{\text{coh}}(X \times X)$ and $\text{D}_{\text{perf}}(A^e)$ under which $\O_\Delta$ corresponds to $A$ with its natural bimodule structure. As $T \boxtimes T^{\vee}$ is a strong generator, $A$ must lie in $\langle A^e \rangle_d$ for some $d$. We have isomorphisms:
\begin{align*}
 \text{Hom}_{A^e}(A, A^e[i]) & \cong \text{Hom}_{X \times X}(\O_\Delta, T \boxtimes T^{\vee}[i]) \\  & \cong \text{Hom}_{X}(\O_X, \Delta^{!}(T \boxtimes T^{\vee})[i])
 \\ & \cong \text{Hom}_{X}(\O_X, \Delta^{*}(T \boxtimes T^{\vee}) \otimes_{\mathcal O_X} \omega_X^{\vee}[i-\text{dim}(X)])
  \\ & \cong \text{Hom}_X(T,T \otimes_{\mathcal O_X} \omega_X^{\vee}[i-\text{dim}(X)]).
\end{align*}
 Since $\tritime(T) = \tritime(A)$ from the equivalence, $\text{D}^{\text{b}}_{\text{coh}}(X)\cong\text{D}_{\text{perf}}(A)$, we can apply Lemma \ref{lem:Extest} to get the upper bound. If $X$ is proper over $k$, then $\op{End}_X(T)$ is finite-dimensional over $k$. If we assume that $k$ is perfect, we can apply Corollary \ref{cor:alldimensionsareequal} and conclude that $\tritime(T) = \dim X + i_0$.
\end{proof}

\begin{corollary}
 Let $X$ be a smooth variety and $T$ be a tilting object in $\ider$. If $\emph{Hom}_X(T,T \otimes_{\mathcal O_X} \omega_X^{\vee}[i])$ is zero for $i$ positive, then the generation time of $T$ is equal to the dimension of $X$ and Conjecture \ref{conj:1} holds for $X$.
\end{corollary}

\begin{corollary}\label{cor:unifbd}
 Let $X$ be a smooth variety and $T$ a tilting sheaf in $\ider$. The generation time of $T$ is bounded above by $2 \dim X$.
\end{corollary}

For the remainder of this section and the paper, we use our new bound on generation time of tilting objects to investigate Conjecture \ref{conj:1} in some examples. Below, we will assume that our base field, $k$, has \textbf{characteristic zero} and is \textbf{algebraically-closed} to achieve sharper statements,  assure all stacks encountered are tame, see Theorem 3.2.b of \cite{AOV08}, and assure we can use the standard definition of the canonical bundle, see Remark \ref{rmk}. We leave the reader to formulate the appropriate statements when $k$ is a more general field.

Before tackling more specific cases, we have the following simple but useful observations:

\begin{lemma}\label{lem:antieffectivecanonical}
Let $X$ be a smooth variety of dimension $n$ such that the anti-canonical divisor is effective.  Any tilting bundle, $T$, has generation time at most $2n-1$.
\end{lemma}
\begin{proof}
Let $Y$ be a subscheme representing the anti-canonical class.  Consider the exact sequence,
\begin{displaymath}
0 \ra T \otimes_{\mathcal O_X} T^{\vee} \to T \otimes_{\mathcal O_X} T^{\vee} \otimes_{\mathcal O_X} \omega^{\vee}_X \to T \otimes_{\mathcal O_X} T^{\vee} \otimes_{\mathcal O_X} \omega_X^{\vee} \otimes_{\mathcal O_X} \O_{Y} \ra 0.
\end{displaymath}
One knows that $T \otimes_{\mathcal O_X} T^{\vee}$ has no higher cohomology by assumption and $T \otimes_{\mathcal O_X} T^{\vee}  \otimes_{\mathcal O_X} \omega^{\vee}_X \otimes_{\mathcal O_X} \O_{Y}$ has no cohomology in degree $n$ since it is supported in dimension $n-1$.  Hence, $T \otimes_{\mathcal O_X} T^{\vee} \otimes_{\mathcal O_X} \omega_X^{\vee}$ does not have cohomology in degree $n$.
\end{proof}

\begin{lemma}\label{lem:highercohomologycanonical}
Let $X$ be a smooth variety over $k$.  Suppose that for some $i$, $\emph{H}^{i}(X,\omega_X^{\vee})$ is nonzero.  Then any tilting bundle $T$ (or more generally any tilting object for which $T \overset{\mathbf{L}}{\otimes}_{\mathcal O_X} T^{\vee}$ contains $\mathcal O_X$ as a summand) has generation time at least $\emph{dim}(X) + i$.
\end{lemma}
\begin{proof}
For a vector bundle, $T \overset{\mathbf{L}}{\otimes}_{\mathcal O_X} T^{\vee} = T {\otimes}_{\mathcal O_X} T^{\vee}$ and, in characteristic zero, the trace map, $T^{\vee} \otimes_{\mathcal O_X} T \to \mathcal O_X$, splits.  Hence $\O_X$ is a summand of $T \overset{\mathbf{L}}{\otimes}_{\mathcal O_X} T^{\vee}$.  Therefore, $\text{H}^{i}(X,\omega_X^{\vee})$ is a summand of $\text{Hom}_X(T,T\otimes_{\mathcal O_X} \omega_X^{\vee}[i])$.
\end{proof}

\subsection{Rational surfaces}\label{subsec:blow-ups}

The following lemma is a useful computational aid:
\begin{lemma} \label{rationalantieffective}
 Let $X$ be a smooth proper surface such that the anti-canonical divisor is effective and the corresponding linear system contains a smooth connected curve, $C$.  Let $D$ be a divisor satisfying: $\emph{H}^i(X,\O(D))=0$ for $i>0$. The line bundle $\O(D-K)$ has no higher cohomology if and only if $(D-K)|_{C}$ is non-trivial and $(K-D) \cdot K \ge 0$. Let $T$ be a tilting object that is a direct sum of line bundles. $T$ has generation time two if and only if $(D-K)|_{C}$ is non-trivial and $(K-D) \cdot K \ge 0$ for every summand, $\O(D)$, of $T \otimes_{\mathcal O_X} T^{\vee}$.
\end{lemma}

\begin{proof}
 Consider the following exact sequence,
\begin{displaymath}
 0 \to \O(D) \to \O(D-K) \to \O_{C}(D-K) \to 0.
\end{displaymath}
 As $\text{H}^i(X,\O(D))=0$ for $i>0$, one has $\text{H}^i(X,\O(D-K)) \cong \text{H}^i(C,\O_{C}(D-K))$ for $i>0$.  Since, by adjunction, $C$ is a smooth curve of genus one, $\text{H}^i(C,\O_{C}(D-K))=0$ for $i>0$ if and only if $(D-K)|_{C}$ is non-trivial and $(D-K)\cdot C = (K-D) \cdot K \ge 0$.

 If $T$ is a tilting bundle that is direct sum of line bundles, then $\op{H}^i(X,\O(D))=0$ for $i>0$ for every summand, $\O(D)$, of $T \otimes_{\mathcal O_X} T^{\vee}$ by definition.
\end{proof}

%When these objects are sums of line bundles they have been nearly classified by L. Hille and M. Perling \cite{HP}.  Later in the section we will summarize the results of Hille and Perling as they apply to the generation time of tilting objects of Del Pezzo surfaces.
Let $\mathcal B_t$ be any blow-up of $\P^2$ at any finite set of distinct points, of cardinality $t$, and $\pi: \mathcal B_t \to \P^2$ be the projection (this is a slight abuse of notation as $\mathcal B_t$ depends on the set and not just the number of points).  Consider the following coherent sheaves:
$$T_1:= \O \oplus  \O(H) \oplus  \O(2H)\oplus \O(E_1) \oplus \cdots \oplus \O(E_t),$$
$$T_2:= \O \oplus  \O(H) \oplus  \O(2H)\oplus  \O_{E_1} \oplus \cdots \oplus \O_{E_t},$$
where $\mathcal{O}(H)=\pi^*\mathcal{O}_{\P^2}(1)$ and $E_1,\ldots,E_t$ are the exceptional divisors.

\begin{proposition} \label{blow-ups}
 If $t \leq 2$ or $t=3$ and the points are not collinear, then the generation time of $T_1$ is two, whereas if $t > 3$ or $t=3$ and the points are collinear, then the generation time of $T_1$ is three. The generation time of $T_2$ is $3$ for all $\mathcal B_t$. Moreover, any tilting bundle on $\mathcal{B}_t$ for $t > 10$ has generation time at least three.
\end{proposition}

\begin{proof}
 We leave the proof that $T_1$ and $T_2$ are tilting as an exercise for the reader, see \cite{KO}.

 First notice that,
\begin{displaymath}
\text{Ext}^2(T_1, T_1\otimes_{\mathcal O_{\mathcal B_t}} \omega_{\mathcal B_t}^\vee)^{\vee} = \text{Hom}(T_1, T_1 \otimes_{\mathcal O_{\mathcal B_t}} \omega_{\mathcal B_t}^2).
\end{displaymath}
 Recall that $\omega_{\mathcal B_t} \cong \mathcal O(-3H+E_1+\cdots+E_t)$. By inspection, $T_1 \otimes_{\mathcal O_{\mathcal B_t}} T_1^{\vee}\otimes_{\mathcal O_{\mathcal B_t}} \omega_{\mathcal B_t}^{2}$ is a sum of line bundles of the form $\O(nH + \sum b_i E_i)$ for some $b_i$ and $n \leq -4$. In particular, these line bundles have no global sections, hence,
\begin{displaymath}
\text{Ext}^2(T_1, T_1 \otimes_{\mathcal O_{\mathcal B_t}} \omega_{\mathcal B_t}^\vee) =0.
\end{displaymath}
 Applying Theorem \ref{main theorem}, we deduce that $\tritime(T_1) \leq 3$.

 Consider the cohomology of $\mathcal{O}(H-E_1-\cdots-E_t)$. The self-intersection of this divisor is $-t+1$. The intersection with the canonical divisor is $t-3$. Thus, by Riemann-Roch, $\chi(\mathcal{O}(H-E_1-\cdots-E_t))$ is negative and $\Ext^1_{\mathcal B_t}(\mathcal{O}(2H),\mathcal{O}(3H-E_1-\cdots-E_t))$ is nonzero unless $t \leq 3$. Hence $\tritime(T_1) = 3$ when $t>3$. In the case, $t=3$, the Euler characteristic of $\mathcal{O}(H-E_1-E_2-E_3)$ is zero and $\mathcal{O}(H-E_1-E_2-E_3)$ has a section if and only if the points are collinear. Hence, $T_1$ has generation time three when the points are collinear.

 Now for $t \leq 3$, we apply Lemma \ref{rationalantieffective}. Write $T_1 = \mathcal L_1 \oplus \cdots \oplus \mathcal L_{t+3}$ for notation and let $\O(D_{ij})= \mathcal L_i \otimes_{\mathcal O_{\mathcal B_t}} \mathcal L_j^{\vee}$.  Then $(K-D_{ij}) \cdot K \geq 0$ with equality if and only if $D_{ij} = -2H$ and $t=3$. We already saw that $\O(H-E_1-E_2-E_3)$ has no higher cohomology when the points are not collinear. By Lemma \ref{rationalantieffective}, $\tritime(T_1) =2$ when $t \leq 2$ or $t=3$ and the points are not collinear.

 Now we consider $T_2$. Some of the Ext-groups we need to compute were covered in the argument for $T_1$. The new ones are $\Ext^s_{\mathcal{B}_t}(\mathcal{O}_{E_i},\omega_{\mathcal{B}_t}^{\vee}\otimes_{\mathcal O_{\mathcal B_t}}\mathcal{O}(mH))$, $\Ext^s_{\mathcal{B}_t}(\mathcal{O}(mH),\mathcal{O}_{E_i}(\omega_{\mathcal{B}_t}^{\vee}))$, $\Ext^s_{\mathcal{B}_t}(\mathcal{O}_{E_i},\mathcal{O}_{E_i}(\omega_{\mathcal{B}_t}^{\vee}))$, and $\Ext^s_{\mathcal{B}_t}(\mathcal{O}_{E_i},\mathcal{O}_{E_j}(\omega_{\mathcal{B}_t}^{\vee}))$. First of all, the cohomology group $\Ext^s_{\mathcal{B}_t}(\mathcal{O}_{E_i},\omega_{\mathcal{B}_t}^{\vee}\otimes_{\mathcal O_{\mathcal B_t}}\mathcal{O}(mH))$ is isomorphic to $\Ext^s_{\mathcal{B}_t}(\mathcal{O}_{E_i},\mathcal{O}(-E_i))$, which is nonzero for $s=1$. Thus, the $\tritime(T_2) \geq 3$ for any $t$.
 
 Now, apply $\Hom_{\mathcal{B}_t}(-,\mathcal{O}(-E_i))$ to the short exact sequence,
\begin{displaymath}
 0 \ra \mathcal{O}(-E_i) \ra \mathcal{O} \ra \mathcal{O}_{E_i} \ra 0.
\end{displaymath}
 Since $\mathcal{O}$ and $\mathcal{O}(-E_i)$ have no higher cohomology, $\Ext^2_{\mathcal{B}_t}(\mathcal{O}_{E_i},\mathcal{O}(-E_i))$ is zero. In addition, the cohomology group $\Ext^s_{\mathcal{B}_t}(\mathcal{O}(mH),\mathcal{O}_{E_i}(\omega_{\mathcal{B}_t}^{\vee}))$ is isomorphic to $\Ext^s_{\mathcal{B}_t}(\mathcal{O},\mathcal{O}_{E_i}(1))$ which is zero for positive $s$.  One also has an isomorphism between $\Ext^s_{\mathcal{B}_t}(\mathcal{O}_{E_i},\mathcal{O}_{E_i}(\omega_{\mathcal{B}_t}^{\vee}))$ and $\Ext^s_{\mathcal{B}_t}(\mathcal{O}_{E_i}(-1),\mathcal{O}_{E_i})$. Apply $\Hom_{\mathcal{B}_t}(-,\mathcal{O}_{E_i})$ to the short exact sequence,
\begin{displaymath}
 0 \ra \mathcal{O} \ra \mathcal{O}(E_i) \ra \mathcal{O}_{E_i}(-1) \ra 0.
\end{displaymath}
 As $\mathcal{O}_{E_i}$ and $\mathcal{O}_{E_i}(1)$ have no higher cohomology, $\Ext^2_{\mathcal{B}_t}(\mathcal{O}_{E_i}(-1),\mathcal{O}_{E_i})$ is zero. Finally, for $\Ext^s_{\mathcal{B}_t}(\mathcal{O}_{E_i},\mathcal{O}_{E_j}(\omega_{\mathcal{B}_t}^{\vee}))$ with $i \neq j$, simply notice that $\Ext^2_{\mathcal{B}_t}(\mathcal{O}_{E_i},\mathcal{O}_{E_j}(\omega_{\mathcal{B}_t}^{\vee})) = \text{Hom}_{\mathcal{B}_t}(\mathcal{O}_{E_j},\mathcal{O}_{E_i}(\omega_{\mathcal{B}_t}^2))^\vee.$  The latter morphism space is clearly zero.  Hence, $\Ext^2_{\mathcal{B}_t}(T_2,T_2 \otimes (\omega_{\mathcal{B}_t}^{\vee}))=0$.  Thus, by Theorem \ref{main theorem}, the generation time of $T_2$ is three for all $t$.

 For the final statement, note that the Euler characteristic of the anti-canonical divisor is $10-t$. Thus, for $t > 10$, $\omega_{\mathcal{B}_t}^{\vee}$ has nontrivial cohomology in degree one.  Applying Lemma \ref{lem:highercohomologycanonical}, we see that the generation time must be at least three.
\end{proof}

\begin{remark}
 From \cite{KO}, the exceptional collections corresponding to $T_1$ and $T_2$ are related by mutation. Thus, generation time is not invariant under mutation. 

 It can be oberved that, in the case of modules over the path algebra of the $A_n$-quiver, if one mutates from the exceptional collection consisting of the indecomposable projective modules to exceptional collection consisting of the simple modules, then one gets the list, $\{1,2,3,\ldots,n-1\}$, of generation times. Moreover, the generation time of any generator must be either zero or a number on this list, see \cite{BFK}.
\end{remark}

% On $\mathcal{B}_1$, we have another full strong exceptional collection $\mathcal{O},\mathcal{O}(E),\mathcal{O}(H+E),\mathcal{O}(2H+2E)$. Set
% \begin{displaymath}
%  T' = \mathcal{O} \oplus \mathcal{O}(E) \oplus \mathcal{O}(H+E) \oplus \mathcal{O}(2H+2E).
% \end{displaymath}
% As in the previous proof, $T'\otimes (T')^{\vee} \otimes \omega_{\mathcal B_1}^2$ has no effective summands so the generation time is at most $3$. However, $\Ext^1_{\mathcal B_1}(\mathcal{O}(2H+2E),\omega_{\mathcal B_1}^{\vee}) \cong \text{H}^1(\mathcal{O}(H-3E))$ is nonzero as $\mathcal{O}(H-3E)$ has negative Euler characteristic, which we can compute from Riemann-Roch. This gives a simple example of a variety with two full strong exceptional collections of line bundles with different generation times.

In \cite{HP}, L. Hille and M. Perling systematically studied the question of when rational surfaces admit full strong exceptional collections consisting of line bundles. We recall one of their definitions:

\begin{definition}
 Let $E_0,\ldots,E_n$ be a strong exceptional collection on a smooth variety, $X$. We say that the collection is \textbf{strongly cyclic} if $E_s,\ldots,E_n,E_0\otimes\omega_X^{\vee},\ldots,E_{s-1}\otimes\omega_X^{\vee}$ is a strong exceptional collection for any $s$. Equivalently, one requires that
\begin{displaymath}
 \emph{Ext}^l_X(E_j,E_i \otimes_{\mathcal O_{X}} \omega_X^{\vee}) = 0 \text{ for } l > 0 \text{ and } i < j.
\end{displaymath}
\end{definition}

Notice that if $X$ is a smooth quasi-projective variety over $k$ and $E_1, \ldots, E_n$ is a full strong exceptional collection of $\der$ such that that generation time of $E_1 \oplus \cdots \oplus E_n$ is equal to the Krull dimension of $X$, then by Theorem \ref{main theorem}, $E_1, \ldots, E_n$ is strongly cyclic. One source of strongly cyclic exceptional collections comes from the following theorem in \cite{HP}:

\begin{theorem}
 Let $X$ be a smooth proper rational surface. If $X$ possesses a full strongly cyclic exceptional collection consisting of line bundles, then $\emph{rk} \ \emph{Pic}(X) \leq 7$. If $X$ is a del Pezzo surface with $\emph{rk} \ \emph{Pic}(X) \leq 7$, then $X$ admits a strongly cyclic exceptional collection consisting of line bundles.
\end{theorem}

\begin{corollary} \label{cor:atmost7}
 Let $X$ be a smooth proper rational surface possessing a strong exceptional collection consisting of line bundles with generation time two, then $\emph{rk} \ \emph{Pic}(X) \leq 7$.
\end{corollary}

Hille and Perling give explicit strong exceptional collections for any del Pezzo surface with Picard rank at most seven.  For a Picard rank seven del Pezzo, we have
\begin{gather*}
\O, \O(E_2), \O(E_1), \O(H-E_3-E_4), \O(H-E_3), \O(H-E_4), \\ \O(2H-E_3-E_4-E_5-E_6), \O(2H-E_3-E_4-E_5), \O(2H-E_3-E_4-E_6)
\end{gather*}
where $E_i$ correspond to the points blown up on $\mathbb{P}^2$, possibly infinitesimally-close. Let $T_3$ be the sum of these line bundles.

\begin{proposition}\label{delpezzo}
 $T_3$ has generation time two.
\end{proposition}
 
\begin{proof}
 By Bertini's theorem, there exists a smooth curve representing $-K$ so we can apply Lemma \ref{rationalantieffective}. After adding $-K$, the intersection of the differences of the line bundles comprising $T_3$ with $-K$ is positive except for $\O(H-E_1-E_2-E_5) \text{ and } \O(H-E_1-E_2-E_6)$. These restrict to the trivial bundle on an anti-canonical curve of genus one if and only if they have sections. However, the points cannot lie on a line as the anti-canonical bundle is ample.
\end{proof}

% From Corollary $4.15$ of \cite{Ha}, we know that a divisor $aH-\sum b_iE_i$ is very ample if and only if $b_i > 0$ for each $i$, $a > b_i+b_j$ for $i \not = j$, and $2a > \sum_{l \not = i} b_l$ for each $i$. By Kodaira vanishing, verifying that $\O(D) \otimes \omega_X^{\vee}$ is ample is enough to show that $\O(D)$ has no higher cohomology. About half of the Ext-vanishing conditions come from \cite{HP}. So we must just check that 
%\begin{displaymath}
% \text{Ext}^l_X(F_i,F_j\otimes \omega_X^{\vee}) = 0 \text{ for } i \leq j \text{ and } l > 0.
%\end{displaymath}

% The worst divisors we must check vanishing for are $\O(3H-2E_2-E_3-E_4-E_5-E_6), \O(4H-E_1-E_2-E_3-2E_4-2E_5-2E_6),$ and 
%$\O(5H-E_1-E_2-2E_3-2E_4-2E_5-2E_5)$. One then verifies that, after tensoring with the anti-canonical bundle, $\O(3H-E_1-E_2-E_3-E_4-E_5-E_6)$, we get ample line bundles.

 %Examples of the worst divisors we must check vanishing for are $\O(3H-2E_2-E_3-E_4-E_5-E_6), \O(4H-E_1-E_2-E_3-2E_4-2E_5-2E_6),$ and $\O(5H-E_1-E_2-2E_3-2E_4-2E_5-2E_5)$. One verifies that each is ample after tensoring with the anti-canonical bundle, $\O(3H-E_1-E_2-E_3-E_4-E_5-E_6)$.

\begin{corollary}
 Conjecture~\ref{conj:1} holds for del Pezzo surfaces with $\emph{rk} \ \emph{Pic}(X) \leq 7$. 
\end{corollary}

\begin{proof}
The above Proposition implies that Conjecture~\ref{conj:1} holds for blow-ups of $\P^2$ at six points in general position.  Any other del Pezzo surface with $\text{rk} \ \text{Pic}(X) \leq 7$ can be obtained as a blow-down.  Suppose $X \to Y$ is a blow-down.  Since $\mathbf{R}\pi_*\O_X \cong \O_Y$, the projection formula yields: $\mathbf{R}\pi_*\circ \mathbf{L}\pi^*(B) \cong B$, for any $B \in \dery$.  In particular, $\mathbf{R}\pi_*$ is a dense functor so we may apply Lemma~\ref{density lemma}.
\end{proof}

We will see in the next section that, due to a result of Van den Bergh, the hypothesis on the Picard rank of the del Pezzo can be dropped if one uses tilting bundles which are not sums of line bundles.

\subsection{Pullback tilting objects}
\begin{proposition} \label{Calabi-Yau}
Suppose $X$ is a smooth variety with $\omega_X$ trivial and possessing a tilting object, $T$. Then the generation time of $T$ is equal to the dimension of $X$.  In particular, Conjecture~\ref{conj:1} holds for $X$. 
\end{proposition}
\begin{proof}
This follows immediately from Theorem~\ref{main theorem}.
\end{proof}

\begin{definition}
Let $X$ be a smooth variety and $\pi: \emph{Tot}(\omega_X) \to X$ be the natural projection from the total space of the canonical bundle.  We say that a tilting bundle, $T$, is \textbf{pullback} if $\pi^*T$ is tilting. If $T$ is the sum of an exceptional collection, we say the corresponding collection is a \textbf{pullback exceptional collection}.
\end{definition}

\begin{lemma}\label{lem:pullbackcrit}
 A tilting object $T$ is pullback if and only if,
\begin{displaymath}
 \emph{Hom}_X(T, T \otimes_{\mathcal O_{X}} \omega_X^{\otimes p}[l]) = 0 \text{ for } l \neq 0 \text{ and } p \leq 0.
\end{displaymath}
 If $T$ is pullback, then the generation time of $T$ equals the dimension of $X$. The generation time of $\pi^*T$ equals $\op{dim } X +1$. 
\end{lemma}

\begin{proof}
 We have $\pi_*\mathcal O_{\op{Tot}(\omega X)} = \bigoplus_{p \leq 0} \omega_X^{\otimes p}$ so, by adjunction,
\begin{displaymath}
 \text{Hom}_{\op{Tot}(\omega_X)}(\pi^*T, \pi^*T[l]) \cong \bigoplus_{p \leq 0} \op{Hom}_{X}(T,T \otimes_{\mathcal O_X} \omega_X^{\otimes p}[l]).
\end{displaymath}

 If $T$ is pullback, it satisfies the conditions of Theorem~\ref{main theorem} with $i_0=0$ so $\tritime(T) = \dim X$. Also, notice that the total space of $\omega_X$ has trivial canonical bundle, hence by Proposition~\ref{Calabi-Yau}, the $\tritime(\pi^*T) = \dim X + 1$.
\end{proof}

%The following type of pullback bundle makes quite a few appearances in the literature (see \cite{BP,Br} for instance):

The following is Proposition 3.3 of \cite{BP}.

\begin{proposition}
 Let $X$ be a smooth variety such that the Grothendieck group, $K_0(X)$, is finitely-generated of rank $\dim X +1$. Any full strong exceptional collection is pullback.
\end{proposition}

Full strong exceptional collections on such varieties have been called simple or geometric, see \cite{BP,Br}.

% \begin{definition}
%  Let $X$ be a smooth variety such that the Grothendieck group, $K_0(X)$, is finitely generated of rank $\dim X +1$.  A full strong pullback exceptional collection on such an $X$ is called a \textbf{simple} (also \textbf{geometric}) exceptional collection.
% \end{definition}

\begin{theorem}
 The following varieties possess simple exceptional collections:  projective spaces, odd-dimensional quadrics, and Fano threefolds of types $V_5$ and $V_{22}$.
\end{theorem}

\noindent The proof of this theorem is due to A. Beilinson \cite{Bei}, M. Kapranov \cite{Kap}, Orlov \cite{O1}, and A. Kuznetsov \cite{Kuz}. Applying Theorem \ref{main theorem}, we get the following:

\begin{corollary}
Conjecture \ref{conj:1} is true for for any variety possessing a simple exceptional collection, in particular for projective spaces, odd-dimensional quadrics, and Fano threefolds of types $V_5$ and $V_{22}$.
\end{corollary}

\begin{lemma}
 Let $X$ be a smooth rational surface such that the anti-canonical divisor is effective and the corresponding linear system contains a smooth connected curve.  If $T$ is a tilting object in $\ider$ with generation time two which is a sum of line bundles, then $T$ is pullback.
\end{lemma}
\begin{proof}
 Let $T = \mathcal L_1 \oplus \cdots \oplus \mathcal L_n$ and $D_{ij} := \mathcal L_i \otimes_{\mathcal O_{X}} \mathcal L_j^{\vee}$.  By Corollary~\ref{cor:atmost7}, we cannot have such a collection unless the rank of the Picard group is less than or equal to seven. For such a rational surface, $K^2 \ge 3$.  By Lemma \ref{rationalantieffective}, $(K-D_{ij}) \cdot K \ge 0$ for all $i,j$.  Therefore $(nK-D_{ij}) \cdot K \ge 3(n-1)$ for all $i,j$.  Applying Lemma \ref{rationalantieffective} one obtains, $\text{H}^k(D_{ij}-nK) =0$ for $k >0, n \ge 2$ and all $i,j$. The cases $n=0$ and $n=1 $ are covered by the assumption that $T$ is tilting and has generation time two. Thus, by Lemma \ref{lem:pullbackcrit}, $T$ is pullback.
\end{proof}

\begin{corollary}
 The tilting objects in section \ref{subsec:blow-ups} of generation time two are pullback.  Namely, the objects, $T_1$ on $\mathcal B_t$ with either $t < 3$ or $t=3$ with the points not collinear, and the object, $T_3$ on a Picard rank seven del Pezzo surface are pullback. 
\end{corollary}

%One also has the following:

%\begin{proposition}\label{prop:pullback}
%Let$\L$ be a line bundle on $X$ and $\pi$ denote the projection from the total space of $\L$ to $X$.  Suppose $T \in \ider$ is a tilting sheaf which is %pullback with respect to $\pi$.  If $\L$ is sufficiently anti-ample, then $T$ is automatically pullback with respect to $\pi$ and the generation time %of $\pi^*T$ is $2\emph{dim}(X)+1$. 
%\end{proposition}

%\begin{proof}
% We have the following isomorphisms:
%\begin{gather*}
% \text{Hom}_{\omega_X}(\pi^*T,\pi^*T[i]) \cong \text{Hom}_X(T, T \otimes \bigoplus_{n \leq 0} {{\mathcal L}^{\otimes n}}[i]) \\
% \text{Hom}_{\omega_X}(\pi^*T, \pi^*T \otimes \omega_{\L}^{\vee}[i]) \cong \text{Hom}_X(T, T \otimes \omega_X^{\vee} \otimes \bigoplus_{n \leq 1} %{{\mathcal L}^{\otimes n}}[i]).
%\end{gather*}
%If $\L$ is sufficiently anti-ample, then $\text{Hom}_X(T,T\otimes\L^n)$ vanishes for $n \leq -1$ and $\text{Hom}_X(T,T\otimes\L[\dim X])$ is nonzero.
%\end{proof}

%\begin{remark}
%In the above situation, one can also show that for any generator $G \in \ider$, one has $\tritime(\pi^*T) \ge \tritime(T)+1$.
%\end{remark}

\begin{remark}
 The condition 
\begin{equation} \label{eqn:strong cyclic}
 \emph{Ext}_X^l(E_i,E_j \otimes_{\mathcal O_{X}} \omega_X^{\vee}) = 0 \text{ for all } i,j \text{ and } l >0
\end{equation} 
 for an exceptional collection $E_0,\ldots,E_n$ can be viewed as first order approximation to being pullback. This condition, as noted previously, is stronger than being strongly cyclic. However, all strongly cyclic exceptional collections in this paper are, in fact, pullback. It would be interesting to ascertain the precise relationship between the notion of strong cyclicity, condition \ref{eqn:strong cyclic}, and the notion of pullback.
\end{remark}

Pullback bundles have also appeared in \cite{V} and are closely related to noncommutative crepant resolutions.  When $X$ is Fano, and $T$ is a pullback tilting bundle, then Proposition 7.2 of loc. cit. states that $\text{End}(\pi^*T)$ is a noncommutative crepant resolution of the anti-canonical ring. Van den Bergh proves the following in Proposition 7.3 of loc. cit.:

\begin{theorem}
 Every del Pezzo surface admits a pullback tilting bundle (which is not necessarily a sum of line bundles).
\end{theorem}

\begin{corollary}
 Conjecture~\ref{conj:1} holds for del Pezzo surfaces.
\end{corollary}

\begin{proof}
 This follows from Lemma \ref{lem:pullbackcrit}.
\end{proof}

%As it turns out, any noncommutative crepant resolution $A$ of a affine Gorenstein variety $S$ will have global dimension equal to the dimension of $S$ (see \cite{SV} Theorem 2.2).  One can also quickly verify this statement using Theorem~\ref{main theorem}.

\subsection{Toric varieties}

Smooth toric varieties are conjecturally a fecund ground for tilting bundles. A. King's conjecture states that any smooth Fano toric variety possesses a full strong exceptional collection. If the dimension is two or if the Picard rank is less than three, the conjecture is true.  However, it is false once one moves to Picard rank three, see \cite{Efi}. On the other hand, in dimension two, a stronger statement is true thanks to further work of Hille and Perling in \cite{HP}.

\begin{theorem}
 Let $X$ be a smooth, proper toric surface. The variety $X$ possesses a strongly cyclic exceptional collection of line bundles if and only if the anti-canonical divisor is nef.
\end{theorem}

Consequently, if the anti-canonical divisor on a toric surface is not nef, we cannot have a strong exceptional collection of line bundles with generation time two. When the anti-canonical divisor is nef, Hille and Perling produce explicit strong cyclic exceptional collections. 

We will not check that each of the exceptional collections produced by Hille and Perling have generation time two; we leave this as an exercise to the reader. We are mainly interested in Conjecture \ref{conj:1} so we content ourselves with a slightly weaker statement:

\begin{proposition}
 Conjecture \ref{conj:1} holds for smooth and proper toric surfaces with nef anti-canonical divisor.
\end{proposition}
\begin{proof}
 We discuss the Picard rank seven toric surface with nef anti-canonical divisor. All others are obtained from this case by blowing down exceptional divisors, except for two of the Picard rank six cases. The proof for these two cases follows along the same lines. The fan for the toric surface with Picard rank seven is found in figure \ref{fig:1}.

\begin{figure}[h]
 \includegraphics{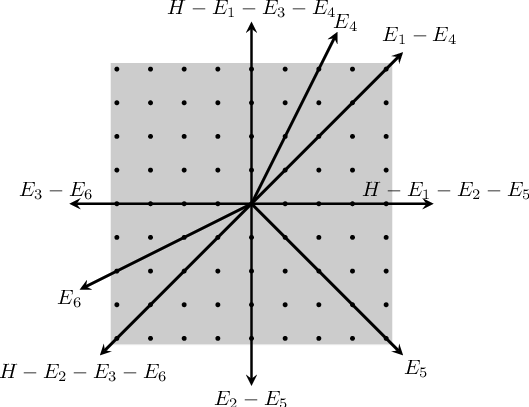}
  \caption{The toric fan of the Picard rank seven toric surface with nef anti-canonical divisor.}
  \label{fig:1}
\end{figure}

% \begin{center}
% \begin{tikzpicture}
%   [scale=.57, vertex/.style={circle,draw=black!100,fill=black!100,thick, inner sep=0.5pt,minimum size=0.5mm}, cone/.style={->,very thick,>=stealth}]
%   \filldraw[fill=black!20!white,draw=white!100]
%     (-4.2,4.2) -- (4.2,4.2) -- (4.2,-4.2) -- (-4.2,-4.2) -- (-4.2,4.2);
%   \draw[cone] (0,0) -- (5.4,0);
%   \draw[cone] (0,0) -- (0,5.4);
%   \draw[cone] (0,0) -- (0,-5.4);
%   \draw[cone] (0,0) -- (-5.4,0);
%   \draw[cone] (0,0) -- (4.5,-4.5);
%   \draw[cone] (0,0) -- (4.5,4.5);
%   \draw[cone] (0,0) -- (-4.5,-4.5);
%   \draw[cone] (0,0) -- (-5.1,-2.55);
%   \draw[cone] (0,0) -- (2.55,5.1);
%   \node at (0,5.8) {$H-E_1-E_3-E_4$};
%   \node at (2.8,5.4) {$E_4$};
%   \node at (5,5) {$E_1-E_4$};
%   \node at (5.8,.4) {$H-E_1-E_2-E_5$};
%   \node at (4.9,-4.9) {$E_5$};
%   \node at (0,-5.8) {$E_2-E_5$};
%   \node at (-5,-5) {$H-E_2-E_3-E_6$};
%   \node at (-5.4,-2.8) {$E_6$};
%   \node at (-5.8,0.4) {$E_3-E_6$};
%  \foreach \x in {-4,-3,...,4}
%   \foreach \y in {-4,-3,...,4}
%   {
%     \node[vertex] at (\x,\y) {};
%   }
% \end{tikzpicture}
% \end{center}
 
 We view this fan as an iterated blow-up of $\P^2$ and have labeled the one dimensional cones accordingly. First, we blow up the three torus invariant points of $\P^2$, then we blow up a single point of each of the three exceptional divisors in a cyclic manner. Precisely, the point on the first exceptional curve corresponds to the tangent direction pointing toward the third, the point on the third exceptional curve corresponds to the tangent direction pointing toward the second, and the point on the second exceptional curve corresponds to the tangent direction pointing toward the first.  Here, we have used $E_1,E_2,$ and $E_3$ to denote the pullbacks of the exceptional divisors of the first round of blow-ups and $E_4,E_5,$ and $E_6$ to denote the infinitesimal blow-ups.
 The exceptional collection we wish to consider is
\begin{gather*}
 \O,\O(E_4),\O(E_2),\O(H-E_3-E_5),\O(H-E_3),\O(H-E_5),\\ \O(2H-E_1-E_3-E_5-E_6),\O(2H-E_1-E_3-E_5),\O(2H-E_3-E_5-E_6).
\end{gather*}
 One can check that there is a smooth and connected divisor in the anti-canonical class. Thus, we can apply Lemma \ref{rationalantieffective}. After adding $-K$, all the differences of these line bundles have positive intersection with the anti-canonical divisor except $\O(H-E_1-E_2-E_4)$ and $\O(H-E_2-E_4-E_6)$, which give zero. The restriction of one of these divisors to an anti-canonical curve of genus one is trivial if and only if it has a section. Examining the configuration of the blow-ups on $\P^2$, we see that neither has sections. 
\end{proof}

% \begin{remark}
% One can make a more precise statement in terms of toric systems.  Suppose $X$ is a rational surface admitting a tilting bundle $T$ which is a sum of line bundles. Furthermore, suppose the anti-canonical bundle is effective and can be represented by a smooth curve.  One can prove that $T$ has generation time two if and only if the corresponding toric system has nef anti-canonical class.
% \end{remark}

\begin{remark}
 For $m \ge 3$, the Hirzebruch surfaces, $\mathbb{F}_m = \P(\O_{\P^1} \oplus \O_{\P^1}(-m))$, have non-nef anti-canonical divisor. By the previous corollary, they cannot possess a full strong exceptional collection of line bundles with generation time two. We will further see that, if $m \geq 4$, then it is also impossible for any tilting bundle to have generation time two. However, Conjecture \ref{conj:1} still holds for $\mathbb{F}_m$, for any $m$, see Proposition \ref{Hirzebrugh}.
\end{remark}

For a larger class of examples, we move to toric stacks. We let $k = \mathbb{C}$ for this. Motivated by King's conjecture, L. Borisov and Z. Hua construct full strong exceptional collections of line bundles for all toric Fano Deligne-Mumford stacks of Picard number at most two or dimension at most two in \cite{BH}. Let us remind the reader of the setup in loc. cit. Let $N$ be a free $\mathbb{Z}$-module and $\Sigma$ a complete fan in $N_{\mathbb{R}}$. Let $v_1,\ldots,v_n$ be primitive rays of the fan. Let $G$ be the following subgroup of $\left( \mathbb{C}^{\times} \right)^n$:
\begin{displaymath}
G:= \lbrace (\lambda_1,\ldots,\lambda_n) \in \left( \mathbb{C}^{\times} \right)^n \ | \ \prod_{i=1}^n \lambda_i^{w (v_i)} = 1 \text{ for all } w \in N^* \rbrace
\end{displaymath}
and $U$ be the following subset of $\mathbb{C}^n$:
\begin{displaymath}
U:= \lbrace (z_1,\ldots,z_n) \in \mathbb{C}^n\ | \ \text{ there is a cone in } \Sigma \text{ containing the } v_i \text{ corresponding to the } z_i=0 \rbrace.
\end{displaymath}
The toric Deligne-Mumford stack $\mathbb{P}_{\Sigma}$ is the stack quotient $[U/G]$. The stack, $\mathbb{P}_{\Sigma}$, is Fano when the $v_i$ are the vertices of a convex simplicial polytope in $N_{\mathbb{R}}$.

We have line bundles $\mathcal{O}(E_i)$ corresponding to the trivial line bundle on $U$ with the $G$-action,
\begin{displaymath}
 (\lambda_1,\ldots,\lambda_n) \cdot (t,z_1,\ldots,z_n) = (\lambda_i t,\lambda_0z_0,\ldots,\lambda_nz_n).
\end{displaymath}
The Picard group of $\mathbb{P}_{\Sigma}$ is generated by the $\mathcal{O}(E_i)$ subject to the relations that $\mathcal{O}(\sum_{i=1}^n w(v_i) E_i) \cong \mathcal{O}$ for $w \in N^*$. 

Borisov and Hua provide a simple sufficient criteria for acyclicity of $\mathcal{O}(\sum_{i=1}^n r_iE_i)$. Let $I$ be a subset of $\lbrace 1,\ldots, n\rbrace$. Corresponding to $I$, we have a simplicial complex $C_I$ whose $k$-simplices are cones of $\Sigma$ containing $k$ of the $v_i$ with $i \in I$. To each $I$ with the reduced homology of $C_I$ nonzero, we have a forbidden point
\begin{displaymath}
 q_I = -\sum_{i \not \in I} E_i \in \op{Pic}_{\mathbb{R}}(\mathbb{P}_{\Sigma})
\end{displaymath}
and a forbidden cone of $\op{Pic}_{\mathbb{R}}(\mathbb{P}_{\Sigma})$
\begin{displaymath}
 F_I = q_I + \sum_{i \in I} \mathbb{R}_{\geq 0} E_i - \sum_{i \not \in I} \mathbb{R}_{\geq 0}E_i \subset \op{Pic}_{\mathbb{R}}(\mathbb{P}_{\Sigma}).
\end{displaymath}
The following is Proposition 4.5 from loc. cit.

\begin{proposition}\label{prop:BH}
 If $\mathcal{O}(\sum_{i=1}^n r_iE_i)$ lies outside all the forbidden cones, then $\op{H}^j(\mathbb{P}_{\Sigma},\mathcal{O}(\sum_{i=1}^n r_iE_i)) = 0$ for $j > 0$.
\end{proposition}

We now use this criteria to prove that the tilting bundles constructed in loc. cit. are pullback.

\begin{proposition} \label{toric examples}
Suppose that $\mathcal X$ is a toric Fano Deligne-Mumford stack of Picard number at most two or dimension at most two. Then, there exists a pullback tilting bundle (which is a sum of line bundles). In particular, Conjecture~\ref{conj:1} holds for $\mathcal{X}$.  %Moreover, if $\mathcal X$ is a variety, the anti-canonical ring of $\mathcal X$ has a noncommutative crepant resolution which is derived equivalent to the total space of $\omega_{\mathcal X}$.
\end{proposition}

\begin{proof}
Recall that $\omega_{\P_{\Sigma}} = \O(-E_1 - \cdots - E_n)$. 

\emph{Case 1}: the Picard rank of $\mathbb{P}_{\Sigma}$ is one, see Proposition 5.1 of loc. cit. The only forbidden cone corresponds to $I = \emptyset$ and
\begin{displaymath}
 F_{I} = - \sum_{i=1}^n E_i - \sum_{i=1}^n \mathbb{R}_{\geq 0}E_i.
\end{displaymath}
Choose an isomorphism $\op{deg}:\op{Pic}_{\mathbb{R}}(\mathbb{P}_{\Sigma}) \to \mathbb{R}$ with the degree of $\omega_{\mathbb{P}_{\Sigma}}$ negative. Let $S$ be the set of line bundles whose degrees are in the interval $(\op{deg}(\omega_{\mathbb{P}_{\Sigma}}),0]$. The bundle, $T:=\bigoplus_{\mathcal{L} \in S} \mathcal L$ is proven to be tilting in loc. cit.. For any two line bundles $\L_1, \L_2 \in S$, we have $\text{deg} (\L_2 \otimes_{\mathcal O_{\mathbb{P}_{\Sigma}}} \L_1^{\vee}) > \text{deg}(\omega_{\mathbb{P}_{\Sigma}})$.  Hence $\text{deg} ( \L_2 \otimes_{\mathcal O_{\mathbb{P}_{\Sigma}}} \L_1^{\vee} \otimes_{\mathcal O_{\mathbb{P}_{\Sigma}}} \omega_{\mathbb{P}_{\Sigma}}^{-n}) =  \text{deg} (\L_2 \otimes_{\mathcal O_{\mathbb{P}_{\Sigma}}} \L_1^{\vee}) - n \cdot \text{deg}(\omega_{\mathbb{P}_{\Sigma}}) > \text{deg}(\omega_{\mathbb{P}_{\Sigma}})$. Since $\op{deg}(F_{I}) \leq \op{deg}(\omega_{\mathbb{P}_{\Sigma}})$, $\L_2 \otimes_{\mathcal O_{\mathbb{P}_{\Sigma}}} \L_1^{\vee} \otimes_{\mathcal O_{\mathbb{P}_{\Sigma}}} \omega_{\mathbb{P}_{\Sigma}}^{-n}$ is acyclic, by Proposition \ref{prop:BH}, and $T$ is pullback.

\emph{Case 2}: the Picard rank of $\mathbb{P}_{\Sigma}$ is two, see Proposition 5.8 of loc. cit. We take $r_i > 0$ with $\sum_{i=1}^n r_i = 1$ and $\sum_{i=1}^n r_iv_i = 0$ and $\alpha_i \in \mathbb{Q}$ nonzero with $\sum_{i=1}^n \alpha_i = 0 $ and $\sum_{i=1}^n \alpha_i v_i = 0$. These define two functions
\begin{align*}
 f: \op{Pic}_{\mathbb{R}}(\mathbb{P}_{\Sigma}) & \to \mathbb{R} \\
 E_i & \mapsto r_i
\end{align*}
and
\begin{align*}
 \alpha: \op{Pic}_{\mathbb{R}}(\mathbb{P}_{\Sigma}) & \to \mathbb{R} \\
 E_i & \mapsto \alpha_i.
\end{align*}
The forbidden cones correspond to $\emptyset,I_+ = \lbrace i \ | \ \alpha_i > 0 \rbrace,I_- = \lbrace i \ | \ \alpha_i < 0 \rbrace$. Let $P \subset \op{Pic}_{\mathbb{R}}(\mathbb{P}_{\Sigma})$ be the polytope defined by $\displaystyle{|f(x)| \leq \frac{1}{2}}$ and $\displaystyle{|\alpha(x)| \leq \frac{1}{2} \sum_{i \in I_+} \alpha_i}$. Take a generic $p \in \op{Pic}_{\mathbb{R}}(\mathbb{P}_{\Sigma})$. The set, $S$, of line bundles we consider consists of those lying in $P + p$.  The bundle, $T:=\bigoplus_{\mathcal L \in S} \mathcal L$ is shown to be tilting in loc. cit.. Take any two line bundles $\L_1, \L_2 \in S$.  Note that $\alpha(\omega_{\mathbb{P}_{\Sigma}}) = -\sum \alpha_i = 0$ and $f(\omega_{\omega_{\mathbb{P}_{\Sigma}}}) = - \sum r_i = -1$. One has, 
\begin{displaymath}
|\alpha(\mathcal{L}_2 \otimes_{\mathcal O_{\mathbb{P}_{\Sigma}}} \mathcal{L}_1^{\vee} \otimes_{\mathcal O_{\mathbb{P}_{\Sigma}}} \omega_{\P_{\Sigma}}^{-n})| = |\alpha (\mathcal L_2 \otimes_{\mathcal O_{\mathbb{P}_{\Sigma}}} \mathcal L_1^{\vee})| < \sum_{i \in I_+} \alpha_i,
\end{displaymath}
 while,
\begin{displaymath}
 \alpha(F_{I_+}) = -\sum_{i \in I_-} \alpha_i  + \sum_{i \in I_+} \mathbb{R}_{\geq 0} \alpha_i - \sum_{i \in I_-} \mathbb{R}_{\geq 0} \alpha_i \geq -\sum_{i \in I_-} \alpha_i = \sum_{i \in I_+} \alpha_i.
\end{displaymath}
Similarly,
\begin{displaymath}
 \alpha(F_{I_-}) \leq  -\sum_{i \in I_+} \alpha_i.
\end{displaymath}
Thus, $\mathcal L_2 \otimes_{\mathcal O_{\mathbb{P}_{\Sigma}}} \mathcal L_1^{\vee} \otimes_{\mathcal O_{\mathbb{P}_{\Sigma}}} \omega_{\P_{\Sigma}}^{-n}$ does not lie in the forbidden cones corresponding to $I_+$ or $I_-$. For $I = \emptyset$, we note that $f(\mathcal L_2 \otimes_{\mathcal O_{\mathbb{P}_{\Sigma}}} \mathcal L_1^{\vee} \otimes_{\mathcal O_{\mathbb{P}_{\Sigma}}} \omega_{\P_{\Sigma}}^{-n}) = f(\mathcal L_2 \otimes_{\mathcal O_{\mathbb{P}_{\Sigma}}} \mathcal L_1^{\vee}) + n \geq -1 + n$ while $f(F_{I}) \leq -1$. Varying $p$, we can move the polytope $P+p$ to assure that $\mathcal L_2 \otimes_{\mathcal O_{\mathbb{P}_{\Sigma}}} \mathcal L_1^{\vee}$ is not sent to $-1$ under $f$, for any $\mathcal L_1,\mathcal L_2 \in S$. Hence $\mathcal L_2 \otimes_{\mathcal O_{\mathbb{P}_{\Sigma}}} \mathcal L_1^{\vee} \otimes_{\mathcal O_{\mathbb{P}_{\Sigma}}} \omega_{\P_{\Sigma}}^{-n}$ does lie in the forbidden cones corresponding to $F_{\emptyset}$.  Since $\mathcal L_2 \otimes_{\mathcal O_{\mathbb{P}_{\Sigma}}} \mathcal L_1^{\vee} \otimes_{\mathcal O_{\mathbb{P}_{\Sigma}}} \omega_{\P_{\Sigma}}^{-n}$ lies in no forbidden cones it is acyclic by Proposition \ref{prop:BH}. Consequently, $T$ is pullback.

\emph{Case 3}: the dimension of $\P_{\Sigma}$ is two, see Proposition 7.2 of loc. cit. We use the projection
\begin{displaymath}
 \pi: \op{Pic}_{\mathbb{R}}(\P_{\Sigma}) \to \op{Pic}_{\mathbb{R}}(\P_{\Sigma})/\mathbb{R}\omega_{\P_{\Sigma}}. 
\end{displaymath}
Again take $r_i > 0$ with $\sum_{i=1}^n r_i = 1$ and $\sum_{i=1}^n r_iv_i = 0$ and consider the function
\begin{align*}
 f: \op{Pic}_{\mathbb{R}}(\mathbb{P}_{\Sigma}) & \to \mathbb{R} \\
 E_i & \mapsto r_i.
\end{align*}
Let $\hat{P}$ be the polytope in $\op{Pic}_{\mathbb{R}}(\P_{\Sigma})/\mathbb{R}\omega_{\P_{\Sigma}}$ defined in Definition 6.10 of loc. cit. We do not recall the precise definition because all we need to know about $\hat{P}$ is that the interior of $\hat{P}$ is disjoint from the images of the forbidden cones $F_I$ if $I \not = \emptyset$, Corollary 6.15 of loc. cit. Let $P$ be the polytope defined by $\displaystyle{|f(x)| \leq \frac{1}{2}}$ and $\displaystyle{\pi(x) \in \frac{1}{2}\hat{P}}$. 
Take $p \in \op{Pic}_{\mathbb{R}}(\P_{\Sigma})$ generic. Let $S$ be the set of line bundles in $P + p$. The bundle, $T:=\bigoplus_{\mathcal L \in S} \mathcal L$ is proven to be tilting in loc. cit..

Take $\mathcal L_1, \mathcal L_2$ from $S$.  As $\pi(\mathcal L_2 \otimes_{\mathcal O_{\mathbb{P}_{\Sigma}}} \mathcal L_1^{\vee} \otimes_{\mathcal O_{\mathbb{P}_{\Sigma}}} \omega_{\P_{\Sigma}}^n) = \pi(\mathcal L_2 \otimes_{\mathcal O_{\mathbb{P}_{\Sigma}}} \mathcal L_1^{\vee}) \not \in \hat{P}$, we see that $\mathcal L_2 \otimes_{\mathcal O_{\mathbb{P}_{\Sigma}}} \mathcal L_1^{\vee} \otimes_{\mathcal O_{\mathbb{P}_{\Sigma}}} \omega_{\P_{\Sigma}}^n$ does not lie in $F_I$ for $I \not = \emptyset$. Also, $f(\mathcal L_2 \otimes_{\mathcal O_{\mathbb{P}_{\Sigma}}} \mathcal L_1^{\vee} \otimes_{\mathcal O_{\mathbb{P}_{\Sigma}}} \omega_{\P_{\Sigma}}^{-n}) \geq -1 + n$.  Again varying $p$, we can assume that $f(\mathcal L_2 \otimes_{\mathcal O_{\mathbb{P}_{\Sigma}}} \mathcal L_1^{\vee}) > -1$, for any $\mathcal L_1,\mathcal L_2 \in S$. Thus, $f(\mathcal L_2 \otimes_{\mathcal O_{\mathbb{P}_{\Sigma}}} \mathcal L_1^{\vee} \otimes_{\mathcal O_{\mathbb{P}_{\Sigma}}} \omega_{\P_{\Sigma}}^{-n})$ is strictly greater than $f(F_{I})$ for $n > 0$. So $\mathcal L_2 \otimes_{\mathcal O_{\mathbb{P}_{\Sigma}}} \mathcal L_1^{\vee} \otimes_{\mathcal O_{\mathbb{P}_{\Sigma}}} \omega_{\P_{\Sigma}}^{-n}$ does not lie in $F_{I}$ if $n > 0$. Again, Proposition \ref{prop:BH} tells us that $T$ is pullback.
\end{proof}

\subsection{Weighted projective spaces and projective bundles}
\label{subsec:weightedproj}

Let $X_{m,n} := \P(\O_{\P^n} \oplus \O_{\P^n}(-m)) \text{ for } m \ge 0$.  Let $\pi: X_{m,n} \to \P^n$ be the projection and $H$ the pullback of the hyperplane section to $X_{m,n}$. Let $S$ be the divisor corresponding to the relative twisting bundle, $\O_{X_{m,n}}(1) = \O(S)$ so that $\pi_*\O(S) = \O_{\P^n} \oplus \O_{\P^n}(-m)$.  Consider the object,
\begin{equation*}\label{T}
 T := \O \oplus \O(H) \oplus \cdots \oplus \O(nH) \oplus \O(S+mH) \oplus \O(S+(m+1)H) \oplus \cdots \oplus \O(S+(m+n)H).
\end{equation*}

\begin{proposition} \label{prop:Hirz}
The object $T$ is a tilting generator.  If $m < n+2$, then the generation time of $T$ is $n+1$, and, if $m \geq n+2$, then the generation time of $T$ is  $2n+1$.  Furthermore, when $m \ge 2n+2$, any tilting bundle on $X_{m,n}$ has generation time equal to $2n+1$.
\end{proposition}

\begin{proof}
From a more general result of Orlov, \cite{O2}, $T$ is a generator. One can check that the indecomposable summands of $T$ comprise a strong exceptional collection by using the computations below.

First we check that the canonical bundle on $X_{m,n}$ is $\mathcal{O}(-2S-(n+1+m)H)$. The Picard group of $X_{m,n}$ is isomorphic to $\mathbf{Z}^2$ with a basis $S$ and $H$ so the canonical divisor is $aS+bH$ for some $a$ and $b$. The divisor $H$ is isomorphic to $X_{m,n-1}$. Restricting $S$ to $H$ gives $S$ and restricting $H$ gives $H$ (allowing for the abuse of notation). Applying adjunction, we have $\mathcal{O}(aS+(b+1)H) \cong \omega_{X_{m,n-1}}$.  Recall that the canonical bundle of the Hirzebruch surface, $\mathbb{F}_m$, is $\mathcal{O}(-2S-(2+m)H)$.  Proceeding by induction, we get $a=-2$ and $b=-n-m-1$.

The space $\Ext^i_{X_{m,n}}(T,T\otimes_{\mathcal O_{X_{m,n}}} \omega_{X_{m,n}}^{\vee})$ is a direct sum of the cohomology groups $\text{H}^i(X_{m,n},\mathcal{O}(aS+bH))$ where either $a=1$ and $1 \leq b \leq 2n+1$, $a=2$ and $m+1 \leq b \leq 2n+m+1$, or $a=3$ and $2m+1 \leq b \leq 2m+2n+1$.  Since $\pi_*$ has no higher direct images when applied to these line bundles,
\begin{align*}
 \text{H}^i(X_{m,n},\mathcal{O}(aS+bH)) & \cong  \text{H}^i(\P^n,\pi_*\mathcal{O}(aS+bH)) \\
& \cong \text{H}^i(\P^n,\text{Sym}^a(\mathcal{O} \oplus \mathcal{O}(-m)) \otimes_{\mathcal O_{\mathbb{P}^n}} \mathcal{O}(b))\\ & \cong \bigoplus_{j=0}^a \text{H}^i(\P^n,\mathcal{O}(-jm+b)).\\
\end{align*}
We will get nonzero higher cohomology if and only if $-am+b \leq -n-1$, i.e. $m \geq n+2$. If $m \geq n+2$, the non-vanishing higher cohomology is concentrated in degree $n$. So, if $m \leq n+1$, $\tritime(T) = \dim X_{m,n} = n+1$, and, if $m \geq n+2$, $\tritime(T) = \dim X_{m,n} + n = 2n+1$. 

In addition,
\begin{displaymath}
 \text{H}^i(X_{m,n},\mathcal{O}(2S+(n+1+m)H) \cong \text{H}^i(\P^n,\mathcal{O}(n+1+m)) \oplus \text{H}^i(\P^n,\mathcal{O}(n+1)) \oplus \text{H}^i(\P^n,\mathcal{O}(n+1-m)).
\end{displaymath}
Since we have a nonzero section of the anti-canonical bundle for any $m$, the generation time must be at most $2n+1$ by Lemma \ref{lem:antieffectivecanonical}.  When $m \geq 2n+2$, we get nonzero cohomology of $\omega_{X_{m,n}}^{\vee}$ in degree $n$. If $T'$ is any tilting object in $\text{D}^\text{b}_{\coh}(X_{m,n})$ with $\mathcal{O}$ a summand of $T' \overset{\mathbf{L}}{\otimes}_{\mathcal O_{X_{m,n}}} T'^{\vee}$, $\tritime(T') \geq 2 \dim X_{m,n} - 1 = 2n+1$ by Lemma \ref{lem:highercohomologycanonical}. %(Note we could have also used Proposition \ref{prop:pullback} to get this lower bound).
\end{proof}

Despite the above proposition, the Rouquier dimension of $\text{D}^\text{b}_{\text{coh}}(X_{m,n})$ is $n + 1$. The Rouquier dimension is achieved by a generator which, in general, is not tilting. Let us denote stacky weighted projective space by $\P(a_0,\ldots,a_n)$.  The category of coherent sheaves on this space is described in \cite{AKO}.  The following lemma is inspired by \cite{AKO}:

\begin{lemma} \label{bundle embedding}
 For $m>n$, $\emph{D}^{\emph{b}}_{\emph{coh}}(X_{m,n})$ is an admissible subcategory of $\emph{D}^{\emph{b}}_{\emph{coh}} (\P (\underbrace{1,\ldots,1}_{n+1},m))$.
\end{lemma}

\begin{proof} The weighted projective stack, $\P(1,\ldots,1,m)$, has as a strong full exceptional collection consisting of the line bundles $\mathcal{O},\mathcal{O}(1),\ldots,\mathcal{O}(m+n)$. Let $E = \bigoplus_{i=0}^{m+n} \mathcal O(i)$. It is straightforward to check that the algebra, $A := \operatorname{End}(E)$, is isomorphic to the path algebra of the following quiver with relations. The set of vertices is $\{0,1,\ldots,m+n\}$. Let $A(i,j)$ denote the set of arrows from vertex $i$ to vertex $j$. We have
\begin{displaymath}
 A(i,j) = \begin{cases} \emptyset & \text{ if }j \not = i+1,i+m \\ \{\alpha_{i,0},\ldots,\alpha_{i,n}\} & \text { if } j=i+1 \\ \{\gamma_i\} & \text{ if } j=i+m \end{cases}
\end{displaymath}
with the relations $\alpha_{i+1,t}\alpha_{i,s} = \alpha_{i+1,s}\alpha_{i,t}$ for each $0 \leq i \leq m+n$ and $0 \leq s,t \leq n$ and $\gamma_{i+1}\alpha_{i,s} = \alpha_{i+n+1,s}\gamma_{i}$ for each $0 \leq i \leq m$ and $0 \leq s \leq n$. Indeed, if we let $x_0,\ldots,x_{n+1}$ denote the coordinates on $\P(1,\ldots,1,m)$, then the isomorphism sends the map, $x_i: \mathcal O(j) \to \mathcal O(j+1)$, to the arrow, $\alpha_{j,i}$, for $0 \leq i \leq n$ and $0 \leq j \leq m+n-1$ and the map, $x_{n+1} : \mathcal O(j) \to \mathcal O(j+n)$, to the arrow, $\gamma_j$, for $0 \leq j \leq m$. See figure \ref{fig:2} for a helpful visualization of the isomorphism in the case of $\P(1,1,4)$.

\begin{figure}[h]
 \includegraphics{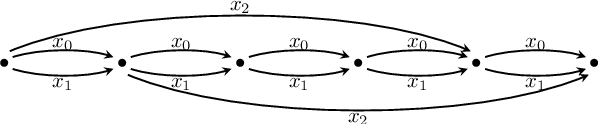}
 \caption{The quiver for a weighted projective plane.}
 \label{fig:2}
\end{figure}

Now take $m > n$, and consider the following strong exceptional collection of line bundles: \begin{displaymath}
\langle \mathcal{O},\mathcal{O}(1),\ldots,\mathcal{O}(n),\mathcal{O}(m),\mathcal{O}(m+1),\ldots,\mathcal{O}(m+n) \rangle.
\end{displaymath}
 Let $E' = \bigoplus_{i=1}^n \mathcal O(i) \oplus \mathcal O(i+m)$ and set $B := \operatorname{End}(E')$. $B$ is isomorphic to the path algebra of the following quiver, $Q$, with relations: the set of vertices is $\{v_0,\ldots,v_n,w_0,\ldots,w_n\}$ and the arrow sets,
\begin{displaymath}
 A(v_i,v_j) = \begin{cases} \emptyset & \text{ if }j \not = i+1 \\ \{a_{i,0},\ldots,a_{i,n}\} & \text { if } j=i+1 \end{cases}
\end{displaymath}
\begin{displaymath}
 A(w_i,w_j)  = \begin{cases} \emptyset & \text{ if }j \not = i+1 \\ \{b_{i,0},\ldots,b_{i,n}\} & \text { if } j=i+1 \end{cases}
\end{displaymath}
\begin{displaymath}
 A(v_i,w_j)  = \begin{cases} \emptyset & \text{ if } j \not = i \text{ or } i\not=n,j\not=0 \\ \{c_i\} & \text{ if } j=i \\ \{p_1,\ldots,p_{\binom{m}{n}}\} & \text { if } i=n,j=0 \end{cases}
\end{displaymath}
\begin{displaymath}
 A(w_i,v_j) = \emptyset.
\end{displaymath}
The relations are $a_{i+1,t}a_{i,s} = a_{i+1,s}a_{i,t}$ for $0 \leq i \leq n-2$ and $0 \leq s,t \leq n$, $b_{i+1,t}b_{i,s} = b_{i+1,s}b_{i,t}$ for $0 \leq i \leq n-2$ and $0 \leq s,t \leq n$, $c_{i+1}a_{i,s}=b_{i,s}c_i$ for $0 \leq i \leq n-1$ and $0 \leq s \leq n$, and $b_{0,s}p_la_{n-1,t} = b_{0,t}p_la_{n-1,s}$ for $0 \leq s,t \leq n$ and $1 \leq l \leq \binom{m}{n}$. Note the set $\op{Hom}(\mathcal O(n),\mathcal O(m))$ is the set of all homogeneous polynomials of degree $m-n$ in $x_0,\ldots,x_n$. Let $q_1,\ldots,q_{\binom{m}{n}}$ be the basis consisting of the set of such monomials ordered lexicographically.

The isomorphism between $B$ and the path algebra of $Q$ modulo these relations sends $x_i: \mathcal O(j) \to \mathcal O(j+1)$ to $a_{j,i}$ for $0 \leq j \leq n-1$ and $0 \leq i \leq n$, $x_i: \mathcal O(j) \to \mathcal O(j+1)$ to $a_{j,i}$ for $m \leq j \leq m+n-1$ and $0 \leq i \leq n$, $x_{n+1}: \mathcal O(j) \to \mathcal O(j+m)$ to $c_j$ for $0 \leq j \leq n$, and $q_l: \mathcal O(n) \to \mathcal O(m)$ to $p_l$ for $1 \leq l \leq \binom{m}{n}$. It is simple to verify this is an isomorphism. See figure \ref{fig:3} for a helpful visualization for $\P(1,1,4)$.

\begin{figure}[h]
 \includegraphics{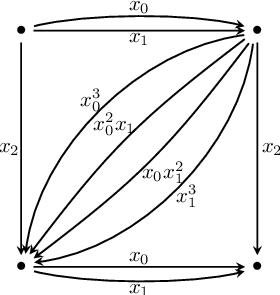}
 \caption{The quiver for a Hirzebruch surface.}
 \label{fig:3}
\end{figure}

We turn back to $X_{m,n}$. Let us compare $\op{End}_{X_{m,n}}(T)$, where, as before,
\begin{displaymath}
 T := \O \oplus \O(H) \oplus \cdots \oplus \O(nH) \oplus \O(S+mH) \oplus \O(S+(m+1)H) \oplus \cdots \oplus \O(S+(m+n)H),
\end{displaymath}
to $Q$. We have isomorphisms
\begin{displaymath}
 \op{Hom}(\O(aH),\O(bH)) \cong \op{H}^0(\P^n,\mathcal O(b-a))
\end{displaymath}
for $0 \leq a,b \leq n$. A similar statement holds for $\op{Hom}(\O(S+(m+a)H),\O(S+(m+b)H))$. We have
\begin{displaymath}
 \op{Hom}(\O(nH),\O(S+mH)) \cong \op{H}^0(\P^n,\O(m-n) \oplus \O(-n)) \cong \op{H}^0(\P^n,\O(m-n))
\end{displaymath}
and
\begin{displaymath}
 \op{Hom}(\O(nH),\O(S+(m+n)H)) \cong \op{H}^0(\P^n,\O(m) \oplus \O) \cong \op{H}^0(\P^n,\O(m)) \oplus \op{H}^0(\P^n,\O).
\end{displaymath}
Abusing notation, we will use $x_0,\ldots,x_n$ to denote the morphisms described above. To get a map from the arrow set of $Q$ to $\op{End}(T)$, we send $a_{j,i}$ to $x_i: \O(jH) \to \O((j+1)H)$ for $0\leq i \leq n$ and $0 \leq j \leq n-1$, $b_{j,i}$ to $x_i: \O(S+jH) \to \O(S+(j+1)H)$ for $0\leq i \leq n$ and $0 \leq j \leq n-1$, $p_l$ to $q_l : \O(nH) \to O(S+mH)$ for $1 \leq l \leq \binom{m}{n}$, and $c_i$ to $1: \O(iH) \to \O(S+(m+i)H)$ for $0 \leq i \leq n$. (Here $1 \in \op{H}^0(\P^n,\O)$). It is straightforward to check this induces a surjective homomorphism, $kQ \to \op{End}(T)$, whose kernel is exactly the relations described above.

From Proposition \ref{tilt}, we have an exact equivalence, $\op{D}^{\op{b}}_{\op{coh}}(X_{m,n}) \cong \op{D}_{\op{perf}}(B)$, and a fully-faithful functor,
\begin{displaymath}
 \bullet \overset{\mathbf{L}}{\otimes}_B E': \op{D}_{\op{perf}}(B) \to \text{D}^{\text{b}}_{\text{coh}}(\P(1,\ldots,1,m)),
\end{displaymath}
which is exact. As both $\dbcoh{X_{m,n}}$ and $\dbcoh{\P(1,\ldots,1,m)}$ have finite Rouquier dimension, Theorem 1.3 of \cite{BV} implies that $\bullet \overset{\mathbf{L}}{\otimes}_B E'$ has both a left and a right adjoint. Consequently, $\dbcoh{X_{m,n}}$ is an admissible subcategory of $\dbcoh{\P(1,\ldots,1,m)}$. 

%The fact that the right adjoint to $\bullet \overset{\mathbf{L}}{\otimes}_B E'$ exists follows from Theorem A.1 of \cite{BV} by a standard argument, see, for example, the proof of Theorem 4.1 of \cite{Nee}. Since both categories possess Serre functors, $- \overset{\mathbf{L}}{\otimes}_B E$ also possesses a left adjoint, making $\text{D}^{\text{b}}_{\text{coh}}(X_{m,n})$ an admissible subcategory of $\text{D}^{\text{b}}_{\text{coh}}(\P(1,\ldots,1,m))$. 

\end{proof}

\begin{lemma} \label{weighted dimension}
 The Rouquier dimension of $\emph{D}^\emph{b}_{\emph{coh}} (\P(a_0,\ldots,a_n))$ is $n$. 
\end{lemma}

\begin{proof}
The weighted projective stack, $\P(a_0,\ldots,a_n)$, is a toric Deligne-Mumford stack of Picard rank one so Propositions \ref{toric examples} applies.
\end{proof}

\begin{remark}
The lemma above can also be realized in two other ways. Firstly, as a more direct application of Theorem~\ref{thm:stackbound}. The relevant computations of cohomology can be found in Theorem $8.1$ of \cite{AZ}, see also the discussion in Section $2$ of \cite{AKO}. Secondly, let $\mu_{r}$ denote the group of $r^{th}$ roots of unity and consider the natural action of $G := \mu_{a_0} \times \dots \times \mu_{a_n}$ on $\P^n$.  One verifies that the terms of the Beilinson resolution have a natural $G$-equivariant structure such that the morphisms are $G$-invariant, see \cite{Ka}.  Hence the category of $G$-equivariant sheaves on $\P^n$, which is equivalent to $\emph{D}^{\emph{b}}_{\emph{coh}}(\P(a_0,\ldots,a_n))$, has a generator of generation time $n$. Either of these methods work over any field of characteristic zero.
\end{remark}

% The following Lemma is standard and due to A. Neeman, we include a proof for completeness.
% \begin{lemma}\label{adjoint density}
%  Let $i: \mathcal A \to \mathcal T$ be a fully-faithful exact functor between triangulated categories. If $i$ possesses a right adjoint, $r: \mathcal T \to \mathcal A$, then $r$ is dense.
% \end{lemma}
% 
% \begin{proof}
%  Let $a$ be an object of $\mathcal A$. Complete the unit morphism, $a \to ri(a)$, to an exact triangle,
% \begin{displaymath}
%  c \overset{\phi}{\to} a \to ri(a) \to c[1].
% \end{displaymath}
%  Apply $i$ to get the triangle,
% \begin{displaymath}
%  ic \overset{i\phi}{\to} ia \to iri(a) \to ic[1].
% \end{displaymath}
%  The map $ia \to iri(a)$ exhibits $ia$ as a summand of $iri(a)$, and, consequently, the map $i\phi$ is zero. As $i$ is fully-faithful, $\phi$ must be zero and $a$ is a summand of $ri(a)$.
% \end{proof}

\begin{proposition} \label{Hirzebrugh}
Conjecture~\ref{conj:1} holds for $X_{m,n}$.
\end{proposition}
\begin{proof}
 For any admissible subcategory, $i:\mathcal A \subset \mathcal T$, with right adjoint, $r: \mathcal T \to \mathcal A$, the unit natural transformation, $\op{Id} \to ri$, is an isomorphism, see the proof of Lemma 3.1 of \cite{Bon}. Consequently, $r$ is essentially surjective. 

 Thus, by Lemma \ref{bundle embedding}, there is an essentially surjective functor $\op{D}^{\op{b}}_{\op{coh}}(\P(1,\ldots,1,m)) \to \op{D}^{\op{b}}_{\op{coh}}(X_{m,n})$. By Lemma~\ref{density lemma}, the Rouquier dimension of $\text{D}^{\text{b}}_{\text{coh}}(X_{m,n})$ is bounded above by the Rouquier dimension of $\text{D}^{\text{b}}_{\text{coh}}(\P(1,\ldots,1,m))$, which is $n+1$ by Lemma~\ref{weighted dimension}.
\end{proof}

\begin{remark}
If one considers noncommutative deformations of weighted projective space $\P_\theta(a_0,\ldots,a_n)$ as in \cite{AKO}, one can obtain the same upper bound, $\emph{dim } \emph{D}^{\emph{b}}_{\emph{coh}}(\P_\theta(a_0,\ldots,a_n)) \leq n$, using Proposition 2.7 of loc. cit. Similarly, for the corresponding noncommutative deformations of $X_{m,n}$, we have $\emph{dim } \emph{D}^{\emph{b}}_{\emph{coh}}(X_{\theta,m,n}) \leq n+1$.  However, as these spaces are noncommutative, a good lower bound is unknown. Recent progress on lower bounds for Rouquier dimension may be useful, see \cite{BeOp, BIKO, Op2}.
\end{remark}

\end{document}